\documentclass[a4paper, 10pt]{amsart}

\usepackage[latin1]{inputenc}
\usepackage[shortlabels]{enumitem}
\usepackage{nameref}
\usepackage[colorlinks=true]{hyperref}
\usepackage[nospace,noadjust]{cite}
 \usepackage[T1]{fontenc}
 \usepackage[final, nopatch=footnote]{microtype}
\usepackage{lastpage, setspace}
\usepackage{graphicx}
\usepackage[all]{xy}
\usepackage{amsmath, amsthm, amssymb, bm, tensor}
\usepackage{listings}
\usepackage[ruled]{algorithm2e}
\usepackage{verbatim}
\usepackage{color}
\usepackage{tikz}
\usetikzlibrary{decorations.markings, decorations.pathreplacing}
\usetikzlibrary{calc}

\usepackage{scalerel}

\makeatletter
\newcommand{\displaybump}{\hbox to \@totalleftmargin{\hfil}}
\makeatother

\setlist[enumerate]{leftmargin=1cm}  

\DeclareMathOperator\Aut{Aut}

\DeclareMathOperator{\Sym}{Sym}
\DeclareMathOperator{\id}{id}
\DeclareMathOperator{\Type}{Type}

\DeclareMathOperator\bbN{\mathbb{N}}

\DeclareMathOperator\bbZ{\mathbb{Z}}

\DeclareMathOperator\calC{\mathcal{C}}
\DeclareMathOperator\calL{\mathcal{L}}

\theoremstyle{definition}
\newtheorem{theorem}{Theorem}[section]
\newtheorem*{theorem*}{Theorem}
\newtheorem{lemma}[theorem]{Lemma}
\newtheorem*{lemma*}{Lemma}

\newtheorem{definition}[theorem]{Definition}
\newtheorem*{definition*}{Definition}

\newtheorem*{remark*}{Remark}
\newtheorem{proposition}[theorem]{Proposition}
\newtheorem*{proposition*}{Proposition}
\newtheorem{example}[theorem]{Example}
\newtheorem*{example*}{Example}
\newtheorem*{sketch of proof}{Sketch of Proof}
\newtheorem*{idea of proof}{Idea of Proof}

\makeatletter
\newtheorem*{rep@theorem}{\rep@title}
\newcommand{\newreptheorem}[2]{%
\newenvironment{rep#1}[1]{%
 \def\rep@title{#2 \scshape \ref{##1}}%
 \begin{rep@theorem}}%
 {\end{rep@theorem}}}
\makeatother

\newreptheorem{theorem}{Theorem}
\newreptheorem{proposition}{Proposition}
\newreptheorem{corollary}{Corollary}
\newreptheorem{lemma}{Lemma}
\newreptheorem{conjecture}{Conjecture}
\newreptheorem{definition}{Definition}

\title[]{Discrete (P)-closed groups acting on trees}
\author{Marcus Chijoff and Stephan Tornier}
\date{\today}

\begin{document}

\begin{abstract}
Reid--Smith parametrised all closed subgroups of automorphism groups of trees with Tits' independence property ($P$) using graph-based combinatorial structures known as local action diagrams. Properties of the group, such as being locally compact, compactly generated or simple, are reflected in its local action diagram. In this article we provide necessary and sufficient conditions on the local action diagram for the associated group to be discrete. 
\end{abstract}

\maketitle

\section*{Introduction}

In the general theory of totally disconnected, locally compact (t.d.l.c.) groups, groups acting on trees play an important role for theoretical and practical reasons.

On the theoretical side, due to the Cayley-Abels graph construction \cite{KM08}, every compactly generated t.d.l.c. group acts vertex-transitively on a connected regular graph $\Gamma$ of finite degree. The action of $G$ on $\Gamma$ lifts to an action on its universal cover, which is a regular tree. Note that when $G$ is discrete, compact generation amounts to finite generation and $\Gamma$ is a classical Cayley graph of $G$.

On the practical side, groups acting on trees constitute a particularly accessible class of t.d.l.c. groups which makes them a popular testing ground for conjectures about general t.d.l.c. groups and the construction of (counter)examples.

Many useful classes of groups acting on trees have been defined and parametrised. This includes Burger--Mozes' \cite{BM00} universal groups $\mathrm{U}(F)$, classifying locally transitive ($P$)-closed groups that contain an edge inversion, Smith's \cite{Smi17} universal groups $\mathrm{U}(F_{1},F_{2})$, classifying locally transitive ($P$)-closed groups preserving the bipartition of a biregular tree, Tornier's \cite{Tor23} generalised universal groups $\mathrm{U}_{k}(F)$, classifying locally transitive ($P_{k}$)-closed ($k\in\mathbb{N}$) groups that contain an involutive edge inversion, as well as Radu's \cite{Rad17} groups, classifying locally alternating or symmetric groups that act boundary-$2$-transitively on sufficiently thick trees. 

In striking recent work, Reid-Smith \cite{RS26} achieved a parametrisation of general ($P$)-closed groups acting on trees in terms of \emph{local action diagrams}, see Section~\ref{sec:local_action_diagrams}.

\begin{repdefinition}{def:local_action_diagram}[{\cite[Definition 3.1]{RS26}}]
A \textbf{local action diagram} $\Delta$ is a triple $(\Gamma, (X_{a})_{a \in A\Gamma}, (G(v))_{v \in V\Gamma})$ consisting of
\begin{enumerate}[(i)]
	\item a connected graph $\Gamma=(V\Gamma,A\Gamma,o,t,r)$,
	\item pairwise disjoint, non-empty sets $X_{a}$ ($a\in A\Gamma$), and
	\item closed subgroups $G(v)\le\Sym(X_{v})$ ($v\in V\Gamma$), where $\smash{X_{v}:=\bigsqcup_{a\in o^{-1}(v)}X_{a}}$, such that the sets $X_{a}$ ($\smash{a\in o^{-1}(v)}$) are precisely the orbits of $G(v)$.
\end{enumerate}
Each $X_{a}$ is a \textbf{colour set}, its elements are \textbf{colours}, and each $G(v)$ is a \textbf{local action}.
\end{repdefinition}

Reid--Smith introduce isomorphisms for local action diagrams and establish a one-to-one correspondence of their isomorphism classes with isomorphism classes of actions $(T,G)$, where $T$ is a tree and $G\le\Aut(T)$ is ($P$)-closed. Given any group $G\le\Aut(T)$, a local action diagram arises by labelling $\Gamma:=G\backslash T$ with data coming from the group action. Conversely, for every local action diagram $\Delta$ there is an associated tree $\mathbf{T}$, termed $\Delta$-tree, and a ($P$)-closed group $\mathbf{U}(\mathbf{T},(G(v))_v)\le\Aut(\mathbf{T})$.

Subsequently, they describe various properties of the group $G$ in terms of its local action diagram, including local compactness, compact generation and simplicity.

In this article, we first show that the action type (see Proposition~\ref{prop:group_types}) of any group $G\le\Aut(T)$  coincides with the action type of its ($P$)-closure and can be determined from its local action diagram in terms of certain partial orientations. This statement is included in \cite[Section 5]{RS26} but its proof is not made explicit.

\begin{reptheorem}{thm:lad_group_types}
Let $\Delta = (\Gamma, (G(v)), (X_{a}))$ be a local action diagram, $\mathbf{T} = (T, \pi, \calL)$ a $\Delta$-tree and $G:=\mathbf U(\mathbf{T},(G(v)))\le\Aut_{\pi}(T)$. Then $G$ is of type
\begin{description}[before={{\renewcommand\makelabel[1]{(\emph{##1})}}}]
	\item[Fixed vertex] if and only if $\Gamma$ is a tree and $\Delta$ contains a single vertex cotree. \\ The vertices of $T$ fixed by $G$ correspond to the single vertex cotrees of $\Delta$.
	\item[Inversion] if and only if $\Delta$ contains a (necessarily unique) cotree consisting of a vertex with a non-orientable loop $a\in A\Gamma$ so that $|X_{a}| = 1$. \\ The unique edge of $T$ inverted by $G$ corresponds to this cotree.
	\item[Lineal] if and only if $\Delta$ contains a cyclic cotree $\Gamma'$ with $|X_{a}| = 1$ for all $a \in A\Gamma'$.
	The two ends of $T$ fixed by $G$ correspond to the cyclic orientations of $\Gamma'$.
	\item[Focal] if and only if $\Delta$ contains a cyclic cotree $\Gamma'$ with a cyclic orientation $O'\subseteq A\Gamma'$ so that $|X_{a}| = 1$ for all $a \in O'$ but there is an $a \in A(\Gamma') \backslash O'$ with $|X_{a}| \geq 2$. The unique end of $T$ fixed by $G$ corresponds to the orientation $O'$.
	\item[Horocyclic] if and only if $\Gamma$ is a tree and $\Delta$ has a unique horocyclic end. \\ The unique end of $T$ fixed by $G$ corresponds to this horocyclic end.
	\item[General] if and only if it is of none of the types above. \\ There is a unique smallest cotree of $\Delta$ which is not of the form indicating the fixed vertex, inversion, lineal, or focal type. It corresponds to the unique minimal subtree of $T$ on which $G$ acts geometrically dense. Moreover, $\Delta$ does not have any horocyclic ends.
\end{description}
\end{reptheorem}

For each type we then characterise discreteness of the group in terms of the local actions in its local action diagram. This strengthens the group-to-diagram correspondence in the spirit of~\cite{RS26}, and allows us to identify those ($P$)-closed groups that belong in the realm of t.d.l.c. groups, rather than the discrete world.

\begin{reptheorem}{thm:discrete}
	Let $\Delta = (\Gamma, (X_{a}), (G(v)))$ be a local action diagram, $\mathbf T = (T, \pi, \calL)$ be a $\Delta$-tree, and $G:=\mathbf U(\mathbf{T},(G(v)))\le\Aut_{\pi}(T)$. If $G$ is of type
	\begin{description}[before={{\renewcommand\makelabel[1]{(\emph{##1})}}}]
		\item[Fixed vertex] then $G$ is discrete if and only if $G(v)$ is trivial for almost all $v\in V\Gamma$, and whenever $X_{v}$ ($v\in V\Gamma$) is infinite then $G(v)$ has a finite base and $G(u)$ is trivial for every $u\in V\Gamma$ such that the arc $a\in o^{-1}(v)$ oriented towards $u$ has an infinite colour set.
		\item[Inversion] then $G$ is discrete if and only if $G(v)$ is trivial for almost all $v\in V\Gamma$, and whenever $X_{v}$ ($v\in V\Gamma$) is infinite then $G(v)$ has a finite base and $G(u)$ is trivial for every $u\in V\Gamma$ such that the arc $a\in o^{-1}(v)$ oriented towards $u$ has an infinite colour set.
		\item[Lineal] then $G$ is discrete if and only if $G(v)$ is trivial for all $v\in V\Gamma$.
		\item[Focal] then $G$ is non-discrete.
		\item[Horocyclic] then $G$ is non-discrete.
		\item[General] then $G$ is discrete if and only if $G(v)$ is semiregular for all $v\in V\Gamma'$ and trivial otherwise, where $\Gamma'$ is the unique smallest cotree of $\Delta$.
	\end{description}
\end{reptheorem}

\subsection*{Acknowledgements}
Both authors owe thanks to Colin Reid for the project idea and helpful discussions, and to an anonymous referee for inspiring Proposition~\ref{prop:type_locally_determined} and pointing out several inaccuracies. The second author acknowledges financial support and hospitality as part of the focus program "Actions of totally disconnected locally compact groups on discrete structures" held at Universität Münster, Germany, where some of this work was completed. Financial support through the ARC DECRA Fellowship DE210100180 is also acknowledged.

\newpage
\section{Preliminaries}\label{sec:preliminaries}

\subsection{Permutation Groups}

Let $\Omega$ be a set. In this section we collect definitions concerning $\Sym(\Omega)$, the group of bijections from $\Omega$ to itself, called the \textbf{symmetric group} of $\Omega$. When $\Omega=\{1,2,\ldots,n\}$, the group $\Sym(\Omega)$ is also denoted by $S_{n}$.

A subgroup $G\le\Sym(\Omega)$ is a \textbf{permutation group}. For $g\in G$ and $\omega\in\Omega$ we write $g\omega$ for the image of $\omega$ under $g$. The \textbf{orbit} of $\omega$ is $G\omega:=\{g\omega\mid g\in G\}$ and the set of orbits is $G\backslash\Omega:=\{G\omega\mid \omega\in\Omega\}$. For $A\subseteq\Omega$ we define $gA:=\{ga\mid a\in A\}$ and $GA:=\{ga\mid a\in A,\ g\in G\}$. When $G$ consists only of the identity bijection we write $G=1$, irregardless of $\Omega$. The \textbf{stabiliser} of $\omega \in \Omega$ is $G_{\omega} := \{g \in G \mid g\omega = \omega\}$. The \textbf{pointwise stabiliser} of $A \subseteq \Omega$ is $G_{A} := \{g \in G \mid \forall a\in A:\ ga = a\}$. The \textbf{setwise stabiliser} of $A$ is $G_{\{A\}}:=\{g \in G \mid \forall a\in A:\ ga\in A\}$. A group $G\le\Sym(\Omega)$ is \textbf{semiregular} if $G_{\omega} = 1$ for all $\omega \in \Omega$. A subset $B \subseteq \Omega$ is a \textbf{base} of $G$ if $G_{B} = 1$.

\subsection{Permutation Topology}

Let $X$ be a set. All left translates of pointwise stabilisers of finite subsets of $X$ form the basis of a topology on $\Sym(X)$, termed the \textbf{permutation topology}. This topology turns $\Sym(X)$ into a Hausdorff, totally disconnected group, see e.g. \cite{KM08} and \cite{Woe91}. Passing to the subspace topology for subgroups of $\Sym(X)$, we in particular have the following discreteness criterion.

\begin{lemma}\label{lem:perm_top_discrete}
Let $X$ be a set and $G\le\Sym(X)$. Then $G$ is discrete if and only if there is a finite set $F\subseteq X$ such that $G_{F}$ is trivial.
\end{lemma}

\begin{proof}
If $G$ is discrete then the singleton set $\{\id\}\subseteq G$ is open and necessarily basic. Hence there is a finite set $F\subseteq X$ so that $G_{F}=\{\id\}$. Conversely, if $G_{F}=\{\id\}$ for some finite set $F\subseteq X$ then $\{\id\}$ is open and hence $G$ is discrete.
\end{proof}

\subsection{Graphs}\label{sec:graphs}

A \textbf{graph} $\Gamma = (V, A, o, t, r)$ consists of a vertex set $V$, an arc set $A$, \textbf{origin} and \textbf{terminus} maps $o,t : A \to V$ and a \textbf{reversal} map $r : A \to A,\ a\mapsto\overline{a}$ so that $r^{2} = \mathrm{id}$ and $o(r(a)) = t(a)$ for all $a\in A$. For $a\in A$, the pair $\{a, \overline{a}\}$ is an \textbf{edge}. An arc $a \in A$ is a \textbf{loop} if $o(a) = t(a)$. Note that a loop $a\in A$ may be either \textbf{orientable}, i.e. $a\neq\overline{a}$, or \textbf{non-orientable}, i.e. $a=\overline{a}$. A graph is \textbf{simple} if it contains no loops and for every $(u,v) \in V^{2}$ there is at most one $a\in A$ such that $(u,v)=(o(a),t(a))$. In this case, an arc $a \in A$ may be referred to by $(o(a), t(a))$. 

A \textbf{subgraph} of a graph $\Gamma = (V, A, o, t, r)$ is a graph $\Gamma' = (V', A', o', t', r')$ such that $V' \subseteq V$, $A' \subseteq A$, $o = o|_{A'}$, $t = t|_{A'}$, and $r = r|_{A'}$. For a subset $V' \subseteq V$ the subgraph \textbf{induced} by $V'$ has vertex set $V'$ and arc set $\{a \in A\Gamma \mid o(a), t(a) \in V'\}$.

Let $\Gamma$ and $\Gamma'$ be graphs. A \textbf{homomorphism} $\theta$ from $\Gamma$ to $\Gamma'$ is a pair of maps $\theta_{V} : V\Gamma\to V\Gamma'$ and $\theta_{A} : A\Gamma\to A\Gamma'$ such that we have $\theta_{V}(o(a))=o(\theta_{A}(a))$ and $r(\theta_{A}(a))=\theta_{A}(r(a))$ for all $a\in A$. If $\theta_{V}$ and $\theta_{A}$ are bijective then $\theta$ is an \textbf{isomorphism}. An isomorphism from $\Gamma$ to itself is an \textbf{automorphism}. The set of all automorphisms of $\Gamma$ is denoted by $\Aut(\Gamma)$ and forms a group under composition, which we equip with the permutation topology for its action on $V$.

For a graph $\Gamma=(V,A,o,t,r)$ and a group $G\le\Aut(\Gamma)$ there is a well-defined \textbf{quotient graph} $G\backslash\Gamma\!:=\!(G\backslash V,G\backslash A,o',t',r')$ with $o'(Ga)\!:=\!Go(a)$, $t'(Ga)\!:=\!Gt(a)$ and $r'(Ga):=Gr(a)$  ($a\in A$), and quotient homomorphism $\pi=\pi_{(\Gamma,G)}:\Gamma\to G\backslash\Gamma$.

For an index set $I \subseteq \mathbb{Z}$, put $I':=\{i \in I \mid i + 1 \in I\}$. A \textbf{path} in $\Gamma$ indexed by $I$ is a sequence of vertices $(v_{i})_{i \in I}$ and edges $(\{a_{i}, \overline{a_{i}}\})_{i \in I'}$ of $\Gamma$ such that $\{a_{i}, \overline{a_{i}}\}$ is an edge between $v_{i}$ and $v_{i+1}$ for all $i \in I'$. The \textbf{length} of a path indexed by $I$ is $|I'|$. A \textbf{directed path} only includes arcs $a_{i}$ such that $o(a_{i}) = v_{i}$ and $t(a_{i}) = v_{i+1}$ for all $i \in I'$. A (directed) path is \textbf{simple} if all its vertices are distinct. A \textbf{ray} in a simple graph $\Gamma$ is a simple path indexed by $\mathbb{N}_{0}$. Two rays are equivalent if there is another ray that contains infinitely many vertices of both of them. An \textbf{end} of $\Gamma$ is an equivalence class of rays. The set of all ends of $\Gamma$ is denoted by $\partial\Gamma$. A \textbf{cycle} of length $1$ in $\Gamma$ is a vertex together with two mutually reverse loops. A cycle of length $2$ in $\Gamma$ is a pair of distinct vertices together with two pairs of mutually reverse arcs connecting them. A cycle of length $n\in\mathbb{N}_{\ge 3}$ in $\Gamma$ is a path indexed by $I = \{0, 1, \dots, n\}$ such that $v_{0} = v_{n}$ but otherwise all vertices are distinct. Given $u,v\in V\Gamma$, the \textbf{distance} $d(u, v)$ between $u$ and $v$ is the minimal length of a path between $u$ and $v$, if one exists, and infinity otherwise. For $v \in V\Gamma$ and $n \geq 1$, the \textbf{ball} $B_{n}(v)$ of radius $n$ around $v$ is the subgraph induced by $\{u\in V\Gamma\mid d(v,u)\le n\}$. A graph is \textbf{connected} if there is a path between any two distinct vertices. A \textbf{tree} is a non-empty, simple, connected graph without cycles. Given a tree $T$ and $a \in AT$ the \textbf{half-tree} $T_{a}$ is the subgraph of $T$ induced by $\{v\in VT\mid d(v,t(a)) < d(v,o(a))\}$. 

An \textbf{orientation} of a graph $\Gamma$ is a subset $O\subseteq A\Gamma$ such that $A\Gamma=O\sqcup\overline{O}$. A \textbf{partial orientation} of $\Gamma$ is a subset $O \subseteq A\Gamma$ such that $O\cap\overline{O}=\varnothing$. The distance of $v\in V\Gamma$ to a subgraph $\Gamma'$ is $d(v, \Gamma'):=\min\{d(v, v')\mid v'\in V\Gamma'\}$. An arc $a \in A\Gamma$ is \textbf{oriented towards} $\Gamma'$ if $d(o(a), \Gamma') > d(t(a), \Gamma')$. It is oriented towards an end $\xi\in\partial T$ if there is a ray $R \in \xi$ containing $a$.

\subsection{Local Action Diagrams}\label{sec:local_action_diagrams}

Property ($P$) was first introduced by Tits \cite{Tit70} to construct simple groups acting on trees. For closed actions, the following generalisation by Banks--Elder--Willis \cite{BEW15} includes Property ($P$) as the case $k=1$. 

\begin{definition}[{\cite[Definition 3.1]{BEW15}}]\label{def:pk_closure}
Let $T$ be a tree and $G\leq\Aut(T)$. The $(P_{k})$-\textbf{closure} ($k\in\bbN_{0}$) of $G$ is
\begin{displaymath}
	G^{(P_{k})} = \{h \in \Aut(T) \mid \forall v \in V(T)\ \forall X\subseteq V(B(v,k)) \text{ finite}\ \exists g \in G:\ g|_{X} = h|_{X}\}.
\end{displaymath}
If $G=G^{(P_{k})}$ then $G$ is $(P_{k})$-\textbf{closed}, or satisfies \textbf{Property} ($P_{k}$).
\end{definition}

Retain the notation of Definition~\ref{def:pk_closure}. When $T$ is locally finite, the quantification over finite sets $X\subseteq V(B(x,k))$ may be replaced by considering just $V(B(v,k))$~itself. In either case, one can show that $\smash{G^{(P_{0})}\ge G^{(P_{1})}\ge G^{(P_{2})}\ge\!\cdots\!\ge G^{(P_{k})}\ge\!\dots\!\ge\overline{G}\ge G}$ as well as $\smash{\left(G^{(P_{k})}\right)^{\hspace{-1.5pt}\scaleto{(P_{k})}{5pt}}=G^{(P_{k})}}$ and $\smash{\bigcap_{k\in\bbN_{0}}G^{(P_{k})}=\overline{G}}$. This suggests to parametrise (closed) subgroups of $\Aut(T)$ by parametrising ($P_{k}$)-closed groups and forming intersections. Finally, we also say \textbf{$(P)$-closed} instead of $(P_{1})$-closed.

In \cite{RS26}, Reid--Smith introduce a powerful parametrisation of ($P$)-closed groups based on the following combinatorial structure. See also the survey article \cite{RS23}.

\begin{definition}[{\cite[Definition 3.1]{RS26}}]\label{def:local_action_diagram}
A \textbf{local action diagram} $\Delta$ is a triple $(\Gamma, (X_{a})_{a \in A\Gamma}, (G(v))_{v \in V\Gamma})$ consisting of
\begin{enumerate}[(i)]
	\item a connected graph $\Gamma$,
	\item pairwise disjoint, non-empty sets $X_{a}$ ($a\in A\Gamma$), and
	\item closed subgroups $G(v)\le\Sym(X_{v})$ ($v\in V\Gamma$), where $\smash{X_{v}:=\bigsqcup_{a\in o^{-1}(v)}X_{a}}$, such that the sets $X_{a}$ ($\smash{a\in o^{-1}(v)}$) are precisely the orbits of $G(v)$.
\end{enumerate}
Each $X_{a}$ is a \textbf{colour set}, its elements are \textbf{colours}, and each $G(v)$ is a \textbf{local action}.
\end{definition}

Reid--Smith introduce a notion of isomorphism for local action diagrams and show in \cite[Theorem 3.3]{RS26} that there is a one-to-one correspondence between isomorphism classes of ($P$)-closed actions ($T,G$), where $T$ is a tree and $G\le\Aut(T)$, and isomorphism classes of local action diagrams. In the following we describe how to pass from a ($P$)-closed action to a local action diagram, and conversely. Passing from a group acting on a tree to a local action diagram is straightforward.

\begin{definition}[{\cite[Definition 3.6]{RS26}}]
Let $T$ be a tree and $G\le\Aut(T)$. Define a local action diagram $\Delta(T,G)=(\Gamma,(X_{a}),(G(v)))$ as follows.
\begin{enumerate}[(i)]
	\item Put $\Gamma:=G\backslash T$ and let $\pi:T\to G\backslash T$ be the quotient map.
	\item For every $v\in V\Gamma$ pick $\tilde{v}\in VT$ with $\pi(\tilde{v})=v$. Given $a\in A\Gamma$ with $o(a)=v$ define $X_{a}:=\{b\in o^{-1}(\tilde{v})\mid\pi(b)=a\}$. Then $X_{v}=o^{-1}(\tilde{v})$.
	\item For every $v\in V\Gamma$ let $G(v)\le\Sym(X_{v})$ be the closure of the permutation group induced on $X_{v}$ by $G_{\tilde{v}}$.
\end{enumerate}
\end{definition}

\newpage
To pass from a local action diagram to a group acting on a tree we first describe how to obtain a tree projecting onto $\Delta$, and then define a group acting on it.

\begin{definition}[{\cite[Definition 3.4]{RS26}}]
Let $\Delta=(\Gamma,(X_{a}),(G(v)))$ be a local action diagram. A $\Delta$-\textbf{tree} $\mathbf{T}$ is a triple $(T,\pi,\calL)$ consisting of a tree $T$, a surjective graph homomorphism $\pi:T\to\Delta$ and a map $\calL:AT\to\bigsqcup_{a\in A\Gamma}X_{a}$, termed $\Delta$-\textbf{colouring}, which for all $v\in VT$ and $a\in o^{-1}(\pi(v))$ restricts to bijections $\calL_{v}:o^{-1}(v)\to X_{\pi(v)}$ and $\calL_{v,a}:\{b\in o^{-1}(v)\mid \pi(b)=a\}\to X_{a}$.
\end{definition}

\begin{lemma}[{\cite[Lemma 3.5]{RS26}}]\label{lem:delta_tree_construction}
Let $\Delta$ be a local action diagram. Then there is a $\Delta$-tree $(T,\pi,\calL)$. Given another $\Delta$-tree $(T', \pi', \calL')$ there is a graph isomorphism $\alpha : T \to T'$ such that $\pi' \circ \alpha = \pi$. 
\end{lemma}

We refer to \cite[Lemma 3.5]{RS26} for a full proof of Lemma~\ref{lem:delta_tree_construction} but include the construction of $\Delta$-trees here for the reader's convenience. See also Example~\ref{ex:delta_tree}.

The following definitions are needed to construct~$T$. Let $\Delta=(\Gamma,(X_{a}),(G(v)))$ be a local action diagram. Given $v\in V\Gamma$ and $c\in X_{v}$ the \textbf{type} $p(c)$ of $c$ is the unique arc $a\in o^{-1}(v)$ with $c\in X_{a}$. A \textbf{coloured path} of length $n\in\mathbb{N}_{0}$ in $\Gamma$ is a sequence $\calC=(c_1, c_2, \dots c_{n})$ of colours such that $o(p(c_{i+1})) = t(p(c_{i}))$ for all $1 \leq i < n$. The \textbf{origin} of $\calC$ is $o(p(c_{1}))$. For all $m\le n$, the path $(c_{1},c_{2},\ldots,c_{m})$ is a \textbf{prefix} of~$\calC$, and any path $(d_{1},d_{2},\ldots,d_{n})$ with $\smash{p(d_{i})=\overline{p(c_{i})}}$ for all $1\le i\le n$ is a \textbf{reverse} of $\calC$.

Pick a base vertex $v_{0}\in V\Gamma$. We inductively define $VT$ as a set of coloured paths in $\Gamma$ originating at $v_{0}$ with chosen reverses. Start with a self-reverse root vertex $()$ and the paths $\{(c)\mid c\in X_{v_{0}}\}$. For every $c\in X_{v_{0}}$ pick $\smash{\overline{c}\in X_{\overline{p(c)}}}$ and put $\smash{\overline{(c)}:=(\overline{c})}$. Now, given $v=(c_{1},c_{2},\ldots,c_{n})$ with reverse $\overline{v}=(d_{1},d_{2},\ldots,d_{n})$ we introduce a new vertex $v_{+c_{n+1}}:=(c_{1},c_{2},\ldots,c_{n},c_{n+1})$ for all $c_{n+1}\in X_{t(p(c_{n}))}\backslash\{d_{n}\}$. Choose  $\smash{d_{n+1}\in X_{\overline{p(c_{n+1})}}}$ and set $\smash{\overline{v_{+c_{n+1}}}}:=(d_{1},d_{2},\ldots,d_{n+1})$.

Depending on the context it is useful to think of vertices or $T$ either as coloured paths or merely as symbols. Define $AT:=AT_{+}\sqcup AT_{-}$, where $AT_{+}$ consists of pairs $(v,w)$ of vertices of $T$ such that $v$ is a prefix of $w$ of length one less than $w$, and $AT_{-}=\{(w,v)\mid (v,w)\in AT_{+}\}$. Given an arc $(v,w)\in AT$, define origin, terminus and reversal by $o(v,w):=v$, $t(v,w):=w$ and $r(v,w):=(w,v)$ respectively. Further, let $\calL(v,w)$ be the last entry of $w$ and $\calL(w,v)$ be the last entry of $\overline{w}$.

We define the graph homomorphism $\pi:T\to\Gamma$ on vertices by $\pi(()):=v_{0}$ and $\pi(v)=t(p(c_{n}))$ for any $v=(c_{1},c_{2},\ldots,c_{n})\in VT$. For arcs, we set $\pi(a):=p(\calL(a))$.

\begin{example}\label{ex:delta_tree}
Consider the local action diagram $\Delta$ on two vertices in Figure~\ref{fig:lad_and_tree}.

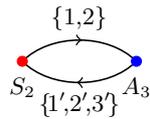
\begin{figure}[ht]
\raisebox{-.5\height}{
\begin{tikzpicture}[scale=1.5]
	\begin{scope}
		[decoration={markings,
			mark=at position 0.52 with {\arrow{>}},
			mark=at position 0.5 with {\node[above]{\small{$\{1,\!2\}$}};}},
		line width=0.6pt]
		\draw [postaction=decorate] (0,0) .. controls (1/4,1/4) and (3/4,1/4) .. (1,0);
	\end{scope}
	\begin{scope}
		[decoration={markings,
			mark=at position 0.52 with {\arrow{>}},
			mark=at position 0.5 with {\node[below]{\small{$\{\!1'\!,\!2'\!,\!3'\!\}$}};}},
		line width=0.6pt]
		\draw [postaction=decorate] (1,0) .. controls (3/4,-1/4) and (1/4,-1/4) .. (0,0);
	\end{scope}
	
	\node (1) at (0,0) {};
	\draw [fill=white] (1) circle [radius=1.25pt];
	\node [below=0.1cm] at (1) {\small{$S_{2}$}};
	
	\node (2) at (1,0) {};
	\draw [fill] (2) circle [radius=1.25pt];
	\node [below=0.1cm] at (2) {\small{$A_{3}$}};
\end{tikzpicture}
}
\hspace{0.5cm}
\raisebox{-.5\height}{
\begin{tikzpicture}[scale=1.4]
	
	\node[fill=white,draw=black,circle,minimum size=4pt,inner sep=0pt] (origin) at (0, 0) {};
	\node[fill,circle,minimum size=4pt,inner sep=0pt] (1) at ($(origin) + (180:1)$) {};
	\node[fill,circle,minimum size=4pt,inner sep=0pt] (2) at ($(origin) + (0:1)$) {}; 
	\node[fill=white,draw=black,circle,minimum size=4pt,inner sep=0pt] (12) at ($(1) + (135:1)$) {};
	\node[fill=white,draw=black,circle,minimum size=4pt,inner sep=0pt] (13) at ($(1) + (225:1)$) {}; 
	\node[fill=white,draw=black,circle,minimum size=4pt,inner sep=0pt] (22) at ($(2) + (45:1)$) {};
	\node[fill=white,draw=black,circle,minimum size=4pt,inner sep=0pt] (23) at ($(2) + (315:1)$) {}; 
	\node[fill,circle,minimum size=4pt,inner sep=0pt] (122) at ($(12) + (135:1)$) {}; 
	\node[fill,circle,minimum size=4pt,inner sep=0pt] (132) at ($(13) + (225:1)$) {}; 
	\node[fill,circle,minimum size=4pt,inner sep=0pt] (222) at ($(22) + (45:1)$) {}; 
	\node[fill,circle,minimum size=4pt,inner sep=0pt] (232) at ($(23) + (315:1)$) {}; 
	\node[circle,minimum size=4pt,inner sep=0pt, anchor=center,rotate=140] (dots1) at ($(122) + (135:0.3) $) {$\dots$}; 
	\node[circle,minimum size=3pt,inner sep=0pt, anchor=center,rotate=220] (dots2) at ($(132) + (225:0.3)$) {$\dots$}; 
	\node[circle,minimum size=3pt,inner sep=0pt, anchor=center,rotate=40] (dots3) at ($(222) + (45:0.3)$) {$\dots$}; 
	\node[circle,minimum size=3pt,inner sep=0pt, anchor=center,rotate=320] (dots4) at ($(232) + (315:0.3)$) {$\dots$}; 
	
	\begin{scope}[decoration={markings, mark=at position 0.53 with {\arrow{>}}},line width=0.6pt]			
		\begin{scope}[label distance=-0.1cm]
			\draw[postaction={decorate},color=black] (origin) to [bend left] node[label=270:$1$] {} (1);
			\draw[postaction={decorate},color=black] (1) to [bend left] node[label=90:$1'$] {} (origin);
			\draw[postaction={decorate},color=black] (origin) to [bend left] node[label=90:$2$] {} (2);
			\draw[postaction={decorate},color=black] (2) to [bend left] node[label=270:$1'$] {} (origin);
		\end{scope}[label distance=-0.25cm]
		\begin{scope}[label distance=-0.25cm]
			\draw[postaction={decorate},color=black] (1) to [bend left] node[label=225:$2'$] {} (12);
			\draw[postaction={decorate},color=black] (1) to [bend left] node[label=315:$3'$] {} (13);	
			\draw[postaction={decorate},color=black] (12) to [bend left] node[label=45:$1$] {} (1);
			\draw[postaction={decorate},color=black] (13) to [bend left] node[label=135:$1$] {} (1);
			\draw[postaction={decorate},color=black] (2) to [bend left] node[label=135:$2'$] {} (22);
			\draw[postaction={decorate},color=black] (2) to [bend left] node[label=45:$3'$] {} (23);
			\draw[postaction={decorate},color=black] (22) to [bend left] node[label=315:$1$] {} (2);
			\draw[postaction={decorate},color=black] (23) to [bend left] node[label=225:$1$] {} (2);
			\draw[postaction={decorate},color=black] (12) to [bend left] node[label=225:$2$] {} (122);
			\draw[postaction={decorate},color=black] (122) to [bend left] node[label=45:$1$] {} (12);
			\draw[postaction={decorate},color=black] (13) to [bend left] node[label=315:$2$] {} (132);
			\draw[postaction={decorate},color=black] (132) to [bend left] node[label=135:$1'$] {} (13);
			\draw[postaction={decorate},color=black] (22) to [bend left] node[above] {\small{$2$}} (222);
			\draw[postaction={decorate},color=black] (222) to [bend left] node[label=315:$1'$] {} (22);
			\draw[postaction={decorate},color=black] (23) to [bend left] node[label=45:$2$] {} (232);
			\draw[postaction={decorate},color=black] (232) to [bend left] node[label=225:$1'$] {} (23);
		\end{scope}
	\end{scope}
\end{tikzpicture}
}
\caption{A local action diagram
\label{fig:lad_and_tree} $\Delta$ and associated $\Delta$-tree.}
\end{figure}

\noindent
Going through the process described above, we see that every $\Delta$-tree is isomorphic to the $(2,3)$-regular tree. Starting at a chosen base vertex, e.g. the central open vertex, note the choices that were made for the reverse labels in Figure~\ref{fig:lad_and_tree}.
\end{example}

Given a $\Delta$-tree $\mathbf{T}=(T,\pi,\calL)$ associated to a local action diagram $\Delta$ we define a ($P$)-closed group acting on it as a subgroup of $\Aut_{\pi}(T):=\{g\in\Aut(T)\mid \pi\circ g=\pi\}$, the group of automorphisms of $T$ that respect $\pi$. To this end, we first define the \textbf{local action} of an automorphism $g\in\Aut_{\pi}(T)$ at a vertex $v\in VT$ using the map
\begin{displaymath}
	\sigma_{\calL,v}:\Aut_{\pi}(T)\to\Sym(X_{\pi(v)}),\ g\mapsto\sigma_{\calL}(g,v):=\calL\circ g\circ\calL|_{o^{-1}(v)}^{-1}.
\end{displaymath}

\begin{definition}[{\cite[Definition 3.8]{RS26}}]\label{def:universal_group}
Let $\Delta=(\Gamma,(X_{a}),(G(v)))$ be a local action diagram and $\mathbf{T}=(T,\pi,\calL)$ be a $\Delta$-tree. The \textbf{universal group of $\mathbf{T}$ with respect to the local actions $(G(v))_{v}$} is
\begin{displaymath}
	\mathbf{U}(\mathbf{T},(G(v))):=\{g\in\Aut_{\pi}(T)\mid \forall v\in VT:\ \sigma_{\calL,v}(g)\in G(\pi(v))\le\Sym(X_{\pi(v)})\}.
\end{displaymath}
\end{definition}

In Example~\ref{ex:delta_tree}, the universal group in the sense of Definition~\ref{def:universal_group} coincides with Smith's (\cite{Smi17}) group $\mathrm{U}(S_{2},A_{3})$ up to an isomorphism of group actions.

By \cite[Theorem 3.9]{RS26}, elements of $\mathbf{U}(\mathbf{T},(G(v)))$ are readily constructed.

\begin{lemma}\label{lem:universal_extension}
Let $\Delta = (\Gamma, (X_{a}), (G(v)))$ be a local action diagram, $\mathbf T = (T, \pi, \calL)$ be a $\Delta$-tree, and $G:=\mathbf U(\mathbf{T},(G(v)))\le\Aut_{\pi}(T)$. Let $w\in VT$ and $\sigma\in G(\pi(w))$. Then there is an element $g\in G_{w}$ with $\sigma_{\calL,v}(g)=\sigma$. Moreover, if $\sigma$ fixes $\calL(b)$ for some $b\in o^{-1}(w)$ then $g$ can be chosen to fix the half-tree $T_{b}$.
\end{lemma}

\begin{proof}
As in the proof of \cite[Theorem 3.9]{RS26}, inductively define $g$ on balls of radius $n\in\mathbb{N}$ around $w$, extending it with trivial local actions in direction $b$.
\end{proof}

\subsubsection{Invariant Subtrees and Ends}

Subtrees and ends of a tree that are invariant under a group action correspond to certain partial orientations of the associated local action diagram. This is described in \cite[Section 5]{RS26} and summarised below.

\begin{definition}[{\cite[Definition 5.2]{RS26}}]
Let $\Gamma$ be a graph. A \textbf{confluent partial orientation} of $\Gamma$ is a partial orientation $O\subseteq A\Gamma$ such that for all $v\in V\Gamma$ we have $|o^{-1}(v)\cap O|\le 1$. A \textbf{strongly confluent partial orientation}, or \textbf{scopo}, of $\Gamma$ is a confluent partial orientation $O$ of $\Gamma$ such that whenever $ o^{-1}(v)\cap O=\{a\}$ for some $v\in V\Gamma$ then $\overline{b}\in O$ for all $b\in o^{-1}(v)\backslash\{a\}$.
\end{definition}

\vspace{-0.4cm}
\begin{figure}[ht]
\begin{tikzpicture}[scale=1.4]
	\node[fill,circle,minimum size=4pt,inner sep=0pt] (origin) at (0, 0) {};
	\node[fill,fill,circle,minimum size=4pt,inner sep=0pt] (0) at ($(origin) + (0:1)$) {};
	\node[fill,fill,circle,minimum size=4pt,inner sep=0pt] (1) at ($(origin) + (72:1)$) {};
	\node[fill,fill,circle,minimum size=4pt,inner sep=0pt] (2) at ($(origin) + (144:1)$) {};
	\node[fill,fill,circle,minimum size=4pt,inner sep=0pt] (3) at ($(origin) + (216:1)$) {};
	\node[fill,fill,circle,minimum size=4pt,inner sep=0pt] (4) at ($(origin) + (288:1)$) {};
	
	\begin{scope}[decoration={markings, mark=at position 0.54 with {\arrow{>}}},line width=0.6pt]			
		\draw[postaction={decorate},color=black,dotted] (origin) to [bend left] node[] {} (0);
		\draw[postaction={decorate},color=black] (0) to [bend left] node[] {} (origin);
		\draw[postaction={decorate},color=black] (origin) to [bend left] node[] {} (1);
		\draw[postaction={decorate},color=black] (1) to [bend left] node[] {} (origin);
		\draw[postaction={decorate},color=black] (origin) to [bend left] node[] {} (2);
		\draw[postaction={decorate},color=black] (2) to [bend left] node[] {} (origin);
		\draw[postaction={decorate},color=black] (origin) to [bend left] node[] {} (3);
		\draw[postaction={decorate},color=black,dotted] (3) to [bend left] node[] {} (origin);
		\draw[postaction={decorate},color=black] (origin) to [bend left] node[] {} (4);
		\draw[postaction={decorate},color=black] (4) to [bend left] node[] {} (origin);
	\end{scope}
\end{tikzpicture}
\hspace{2cm}
\begin{tikzpicture}[scale=1.4]
	\node[fill,circle,minimum size=4pt,inner sep=0pt] (origin) at (0, 0) {};
	\node[fill,fill,circle,minimum size=4pt,inner sep=0pt] (0) at ($(origin) + (0:1)$) {};
	\node[fill,fill,circle,minimum size=4pt,inner sep=0pt] (1) at ($(origin) + (72:1)$) {};
	\node[fill,fill,circle,minimum size=4pt,inner sep=0pt] (2) at ($(origin) + (144:1)$) {};
	\node[fill,fill,circle,minimum size=4pt,inner sep=0pt] (3) at ($(origin) + (216:1)$) {};
	\node[fill,fill,circle,minimum size=4pt,inner sep=0pt] (4) at ($(origin) + (288:1)$) {};
	
	\begin{scope}[decoration={markings, mark=at position 0.54 with {\arrow{>}}},line width=0.6pt]
		\draw[postaction={decorate},color=black,dashed] (origin) to [bend left] node[] {} (0);
		\draw[postaction={decorate},color=black] (0) to [bend left] node[] {} (origin);
		\draw[postaction={decorate},color=black] (origin) to [bend left] node[] {} (1);
		\draw[postaction={decorate},color=black,dashed] (1) to [bend left] node[] {} (origin);
		\draw[postaction={decorate},color=black] (origin) to [bend left] node[] {} (2);
		\draw[postaction={decorate},color=black,dashed] (2) to [bend left] node[] {} (origin);
		\draw[postaction={decorate},color=black] (origin) to [bend left] node[] {} (3);
		\draw[postaction={decorate},color=black,dashed] (3) to [bend left] node[] {} (origin);
		\draw[postaction={decorate},color=black] (origin) to [bend left] node[] {} (4);
		\draw[postaction={decorate},color=black,dashed] (4) to [bend left] node[] {} (origin);
	\end{scope}
\end{tikzpicture}
\caption{Confluent (dotted) versus strongly confluent (dashed).}
\label{fig:copo_scopo}
\end{figure}
\vspace{-0.2cm}

For local action diagrams, the definition of a (strongly) confluent partial orientation takes the size of the colour sets into account.

\begin{definition}
Let $\Delta\!=\!(\Gamma,(X_{a}),(G(v)))$ be a local action diagram. A (\textbf{strongly}) \textbf{confluent partial orientation} $O$ of $\Delta$ is a (strongly) confluent partial orientation of $\Gamma$ such that $|X_{a}| = 1$ for all $a \in O$. 
\end{definition}

Similarly, we single out certain ends of local action diagrams.

\begin{definition}
Let $\Delta\!=\!(\Gamma,\! (X_{a}),\! (G(v)))$ be a local action diagram. A \textbf{horocyclic end} of $\Delta$ is an end $\xi\!\in\!\partial\Gamma$ such that $|X_{a}|\!=\!1$ for all arcs $a\in A\Gamma$ oriented towards $\xi$.
\end{definition}

Strongly confluent partial orientations of connected graphs fall into three distinct types, see Definition~\ref{def:scopo_types}. To describe them we introduce further notation.

Let $\Gamma$ be a graph. A directed path with vertices $(v_0, v_1, \dots, v_{n})$  ($n \geq 2$) in $\Gamma$ is \textbf{backtracking} if $v_{i} = v_{i+2}$ for some $0 \leq i \leq n-2$. Given a subgraph $\Gamma' \subseteq \Gamma$, a~\textbf{projecting path} from $v \in V(\Gamma)$ to $\Gamma'$ is a non-backtracking path starting at $v$, terminating at some $v' \in V(\Gamma')$, and such that $v_{i} \notin V(\Gamma')$ for all $i < n$. Note that if $v \in V(\Gamma')$ then $(v)$ is a projecting path from $v$ to $\Gamma'$.

\begin{definition}
Let $\Delta=(\Gamma,(X_{a}),(G(v)))$ be a local action diagram. A \textbf{cotree} $\Gamma'$ of $\Gamma$ is a non-empty, induced subgraph such that for every $v \in V\Gamma\backslash V\Gamma'$ there is a unique projecting path from $v$ to $\Gamma'$. A \textbf{cotree} $\Gamma'$ of $\Delta$ is a cotree of $\Gamma$ such that $|X_{a}|=1$ for all arcs $a\in A\Gamma$ that belong to a projecting path to $\Gamma'$.
\end{definition}

Note that a cotree of a graph is a subgraph whose complement has no cycles. Also, due to uniqueness of projecting paths, a cotree in a connected graph is connected.
A finite connected graph $\Gamma=(V,A,o,t,r)$ is \textbf{cyclic} if it can be written as a cycle of some length $n\in\mathbb{N}$ as per Section~\ref{sec:graphs}. A \textbf{cyclic orientation} of a cyclic graph $\Gamma$ is an orientation of the graph which includes only one element of $o^{-1}(v)$ for all $v \in \Gamma$. Every cycle graph has exactly two cyclic orientations.

\begin{definition}\label{def:scopo_types}
Let $\Gamma$ be a connected graph. A scopo $O$ of $\Gamma$ is of type
\begin{enumerate}[label=(\alph*)]
	\item if there is a cotree $\Gamma'$ of $\Gamma$ such that $O$ consists of all arcs $a\in A\Gamma$ such that $o(a)\notin V\Gamma'$ and $a$ lies on a projecting path from $o(a) $ to $\Gamma'$. Write $O=O_{\Gamma'}$.
	\item if there is a cyclic cotree $\Gamma'$ of $\Gamma$ such that $O$ is the union of $O_{\Gamma'}$ with a cyclic orientation of $\Gamma'$. Write $O=O_{\Gamma'}^{+}$.
	\item if $\Gamma$ is a tree and there is an end $\xi\in\partial\Gamma$ such that $O$ consists of all arcs $a\in A\Gamma$ oriented towards $\xi$. Write $O=O_{\xi}$.
\end{enumerate}
\end{definition}

The types of Definition~\ref{def:scopo_types} account for all possible scopos by \cite[Section~5]{RS26}. In particular, if $\Gamma$ is a tree, then only type (a) and (c) occur and so scopos correspond to cotrees $\Gamma'$ and ends $\xi$. Moreover, invariance of the cotree or end under a group action is equivalent to invariance of the corresponding scopo.

\begin{lemma}[{\cite[Lemma 5.7]{RS26}}]\label{lem:scopo_preimage}
Let $T$ be a tree, $G \leq \Aut(T)$, $\Delta=\Delta(T,G)$ the associated local action diagram, and $\pi = \pi_{(T, G)}$. Then $O$ is a $G$-invariant scopo of $T$ if and only if $O = \pi^{-1}(O')$ for a scopo $O'$ of $\Delta$. 
\end{lemma}

In particular, Lemma~\ref{lem:scopo_preimage} shows that there is a correspondence between the scopos of a local action diagram and the invariant scopos of the associated action. To identify the corresponding invariant structure, see Proposition~\ref{prop:scopo_structure_correspondence}, we consider the scopo's \emph{attractor}. Given a confluent partial orientation $O$ of a graph $\Gamma$ define $f_{O}:V\Gamma\to V\Gamma$ on $v\in V\Gamma$ as follows: if  $o^{-1}(v) \cap O = \{a\}$ then $f_{O}(v):=t(a)$, and otherwise $f_{O}(v):=v$. Repeated application of $f_{O}$ to $v\in V\Gamma$ results either in an eventually periodic or an aperiodic sequence $(v, f_{O}(v), f_{O}^{2}(v), \dots)$. If it is eventually periodic, let $z_{O}(v)$ be the associated periodic orbit, and otherwise let $z_{O}(v)$ be the associated end. We then define the \textbf{attractor} of $O$ to be $K(O) = \bigcup_{v \in V(\Gamma)}z_{O}(v)$. In words, the attractor is the union of vertices of $\Gamma$ belonging to a periodic orbit of $f_{O}$ and all ends of $\Gamma$ associated to an aperiodic orbit of $f_{O}$.

\begin{theorem}[{\cite[Theorem 5.10]{RS26}}]\label{thm:attractor}
Let $\Gamma$ be a connected graph, $O$ a scopo of $\Gamma$, and $K$ the attractor of $O$. Then exactly one of the following holds.
\begin{enumerate}[label=(\alph*)]
	\item There is a cotree  $\Gamma'$ of $\Gamma$ such that $V\Gamma' = K$ and $O = O_{\Gamma'}$. 
	\item There is a cyclic cotree $\Gamma'$ of $\Gamma$ such that $V\Gamma' = K$ and $O = O_{\Gamma'}^{+}$ for a cyclic orientation of $\Gamma$.
	\item There is an end $\xi$ of $\Gamma$ such that $K = \{\xi\}$ and $O = O_{\xi}$.
\end{enumerate}
\end{theorem}

\newpage
\section{The action type of a local action diagram}\label{sec:action_type}

It is well-known, likely to be attributed to Tits \cite[Proposition~3.4]{Tit70}, that group actions on trees fall into six mutually exclusive types, see Proposition~\ref{prop:group_types}. We also reference Reid--Smith \cite[Theorem~2.5]{RS26} who include a full proof. In this section, we show how to determine the action type of any group acting on a tree from its local action diagram alone. In particular, the type of a group acting on a tree is a \emph{locally determined global property} in the sense of Reid--Smith \cite[Section~8]{RS26}, i.e., a property that is determined by the ($P$)-closure.

\begin{proposition}[{\cite[Theorem 2.5]{RS26}}]\label{prop:group_types}
	Let $G$ be a group acting on a tree~$T$. Then $G$ belongs to exactly one of the following types: the group $G$ either
	\begin{description}[before={{\renewcommand\makelabel[1]{(\emph{##1})}}}]
		\item[Fixed vertex] fixes a (not necessarily unique) vertex,
		\item[Inversion] preserves a unique edge and contains an inversion of that edge,
		\item[Lineal] fixes exactly two ends and translates the line between them,
		\item[Focal] fixes a unique end and contains a translation towards this end,
		\item[Horocyclic] fixes a unique end but no vertices, and acts without translation, or
		\item[General] acts with translation and does not fix any end. 
	\end{description}
\end{proposition}

Proposition~\ref{prop:group_types} is summarised by the following decision tree.

\begin{figure}[ht]
\centering
\begin{tikzpicture}[
	xscale=1.05,
	yscale=0.75,
	node/.style={%
		draw,
		rectangle,
		align=center,
	},
	]	
	\newdimen\nodeDist
	\nodeDist=35mm	
	\node [node] (A) {\small Fixes a \\ \small vertex?};
	\path (A) ++(-135:\nodeDist) node [] (B) {Fixed Vertex};
	\path (A) ++(-45:\nodeDist) node [node] (C) {\small Preserves \small  an edge?};
	\path (C) ++(-135:\nodeDist) node [] (D) {Inversion};
	\path (C) ++(-45:\nodeDist) node [node] (E) {\small Fixes \\ \small an end?};
	\path (E) ++(-135:\nodeDist) node [node] (F) {\small Fixes more \\ \small than one end?};
	\path (E) ++(-45:\nodeDist) node [] (G) {General};
	\path (F) ++(-135:\nodeDist) node [] (H) {Lineal};
	\path (F) ++(-45:\nodeDist) node [node] (I) {\small Contains \\\small  translations?};
	\path (I) ++(-135:\nodeDist) node [] (J) {Focal};
	\path (I) ++(-45:\nodeDist) node [] (K) {Horocyclic};
		
	\newdimen\labelDist
	\labelDist=0.25pt
	\draw (A) -- (B) node [left,pos=\labelDist,xshift=-0.1cm] {Yes}(A);
	\draw (A) -- (C) node [right,pos=\labelDist,xshift=0.1cm] {No}(A);
	\draw (C) -- (D) node [left,pos=\labelDist,xshift=-0.1cm] {Yes}(A);
	\draw (C) -- (E) node [right,pos=\labelDist,xshift=0.1cm] {No}(A);
	\draw (E) -- (F) node [left,pos=\labelDist,xshift=-0.1cm] {Yes}(A);
	\draw (E) -- (G) node [right,pos=\labelDist,xshift=0.1cm] {No}(A);
	\draw (F) -- (H) node [left,pos=\labelDist,xshift=-0.1cm] {Yes}(A);
	\draw (F) -- (I) node [right,pos=\labelDist,xshift=0.1cm] {No}(A);
	\draw (I) -- (J) node [left,pos=\labelDist,xshift=-0.1cm] {Yes}(A);
	\draw (I) -- (K) node [right,pos=\labelDist,xshift=0.1cm] {No}(A);
\end{tikzpicture}
\caption{Decision tree of Proposition~\ref{prop:group_types}.}%
\label{fig:decision tree}%
\end{figure}

\begin{example}
All six types of Proposition~\ref{prop:group_types} appear among $(P)$-closed groups. We give prototypical examples along with their local action diagrams in the case of the $3$-regular tree $T:=T_{3}$. Let $x\in VT$, $a\in AT$ and $\omega,\omega'\in\partial T$. Recall that the end stabiliser $\Aut(T)_{\omega}$ splits as a semidirect product $\Aut(T)_{\omega}\cong\bbZ\ltimes H$, where $H$ consists of all elements of $\Aut(T)_{\omega}$ that fix a vertex, and $\mathbb{Z}$ is generated by a translation of length~$1$ towards the end $\omega$.

\newdimen\ybase
\ybase=-0.08cm
\centerline{
\begin{tabular}{c|c|c}
	Type & $G$ & $\Delta(T,G)$ \\ \hline
	(\emph{Fixed Vertex}) & $\Aut(T)_{x}$ & 	\begin{tikzpicture}[baseline=(0.base),scale=1.5]
		\node (0) at (0,\ybase) {};	
		\begin{scope}
			[line width=0.6pt, decoration={markings,
				mark=at position 0.53 with {\arrow{>}},
				mark=at position 0.5 with {\node[above]{\small{$\{\!1,\!2,\!3\!\}$}};}
			}]
			\draw [black, postaction=decorate] (0,0) .. controls (1/4,1/4) and (3/4,1/4) .. (1,0);
		\end{scope}
		
		\begin{scope}
			[line width=0.6pt, decoration={markings,
				mark=at position 0.53 with {\arrow{>}},
				mark=at position 0.5 with {\node[below]{\small{$\{\!1\!\}$}};}
			}]
			\draw [black, postaction=decorate] (1,0) .. controls (3/4,-1/4) and (1/4,-1/4) .. (0,0);
		\end{scope}
		
		\begin{scope}
			[line width=0.6pt, decoration={markings,
				mark=at position 0.53 with {\arrow{>}},
				mark=at position 0.5 with {\node[above]{\small{$\{\!2,\!3\!\}$}};}
			}]
			\draw [black, postaction=decorate] (1,0) .. controls (5/4,1/4) and (7/4,1/4) .. (2,0);
		\end{scope}
		
		\begin{scope}
			[line width=0.6pt, decoration={markings,
				mark=at position 0.53 with {\arrow{>}},
				mark=at position 0.5 with {\node[below]{\small{$\{\!1\!\}$}};}
			}]
			\draw [black, postaction=decorate] (2,0) .. controls (7/4,-1/4) and (5/4,-1/4) .. (1,0);
		\end{scope}
		
		\begin{scope}
			[line width=0.6pt, line width=0.6pt, decoration={markings,
				mark=at position 0.53 with {\arrow{>}},
				mark=at position 0.5 with {\node[above]{\small{$\{\!2,\!3\!\}$}};}
			}]
			\draw [black, postaction=decorate] (2,0) .. controls (9/4,1/4) and (11/4,1/4) .. (3,0);
		\end{scope}
		
		\begin{scope}
			[line width=0.6pt, line width=0.6pt, decoration={markings,
				mark=at position 0.53 with {\arrow{>}},
				mark=at position 0.5 with {\node[below]{\small{$\{\!1\!\}$}};}
			}]
			\draw [black, postaction=decorate] (3,0) .. controls (11/4,-1/4) and (9/4,-1/4) .. (2,0);
		\end{scope}
		
		\draw [black, dashed] (3,0) .. controls (13/4,1/4) and (15/4,1/4) .. (4,0);
		\draw [black, dashed] (3,0) .. controls (13/4,-1/4) and (15/4,-1/4) .. (4,0);
		
		\draw [color=white, fill=white] (4,0.5) rectangle (3.5,-0.5);
		
		\node (1) at (0,0) {};
		\draw [fill] (1) circle [radius=1.25pt];
		\node [below] at (1) {\small{$S_{3}$}};
		
		\node (2) at (1,0) {};
		\draw [fill] (2) circle [radius=1.25pt];
		\node [below] at (2) {\small{$C_{2}$}};
		
		\node (3) at (2,0) {};
		\draw [fill] (3) circle [radius=1.25pt];
		\node [below] at (3) {\small{$C_{2}$}};
		
		\node (4) at (3,0) {};
		\draw [fill] (4) circle [radius=1.25pt];
		\node [below] at (4) {\small{$C_{2}$}};
	\end{tikzpicture} \\
	(\emph{Inversion}) & $\Aut(T)_{\{a,\overline{a}\}}$ & \begin{tikzpicture}[baseline=(0.base), scale=1.5]
		\node (0) at (1,\ybase) {};
		
		\begin{scope}
			[line width=0.6pt, line width=0.6pt, decoration={markings,
				mark=at position 0.5 with {\node[label={[black, label distance=-2mm]{180}:\small{$\{\!1\!\}$}}]{};}					
			}]
			\draw [black, postaction=decorate] (0,0) .. controls (135:0.9) and (225:0.9) .. (0,0);
		\end{scope}
			
		\begin{scope}
			[line width=0.6pt, line width=0.6pt, decoration={markings,
				mark=at position 0.53 with {\arrow{>}},
				mark=at position 0.5 with {\node[above]{\small{$\{\!2,\!3\!\}$}};}
			}]
			\draw [black, postaction=decorate] (0,0) .. controls (1/4,1/4) and (3/4,1/4) .. (1,0);
		\end{scope}
		
		\begin{scope}
			[line width=0.6pt, line width=0.6pt, decoration={markings,
				mark=at position 0.53 with {\arrow{>}},
				mark=at position 0.5 with {\node[below]{\small{$\{\!1\!\}$}};}
			}]
			\draw [black, postaction=decorate] (1,0) .. controls (3/4,-1/4) and (1/4,-1/4) .. (0,0);
		\end{scope}
		
		\begin{scope}
			[line width=0.6pt, line width=0.6pt, decoration={markings,
				mark=at position 0.53 with {\arrow{>}},
				mark=at position 0.5 with {\node[above]{\small{$\{\!2,\!3\!\}$}};}
			}]
			\draw [black, postaction=decorate] (1,0) .. controls (5/4,1/4) and (7/4,1/4) .. (2,0);
		\end{scope}
		
		\begin{scope}
			[line width=0.6pt, line width=0.6pt, decoration={markings,
				mark=at position 0.53 with {\arrow{>}},
				mark=at position 0.5 with {\node[below]{\small{$\{\!1\!\}$}};}
			}]
			\draw [black, postaction=decorate] (2,0) .. controls (7/4,-1/4) and (5/4,-1/4) .. (1,0);
		\end{scope}
		
		\begin{scope}
			[line width=0.6pt, line width=0.6pt, decoration={markings,
				mark=at position 0.53 with {\arrow{>}},
				mark=at position 0.5 with {\node[above]{\small{$\{\!2,\!3\!\}$}};}
			}]
			\draw [black, postaction=decorate] (2,0) .. controls (9/4,1/4) and (11/4,1/4) .. (3,0);
		\end{scope}
		
		\begin{scope}
			[line width=0.6pt, line width=0.6pt, decoration={markings,
				mark=at position 0.53 with {\arrow{>}},
				mark=at position 0.5 with {\node[below]{\small{$\{\!1\!\}$}};}
			}]
			\draw [black, postaction=decorate] (3,0) .. controls (11/4,-1/4) and (9/4,-1/4) .. (2,0);
		\end{scope}
		
		\draw [black, dashed] (3,0) .. controls (13/4,1/4) and (15/4,1/4) .. (4,0);
		\draw [black, dashed] (3,0) .. controls (13/4,-1/4) and (15/4,-1/4) .. (4,0);
		
		\draw [color=white, fill=white] (4,0.5) rectangle (3.5,-0.5);
		
		\node (1) at (0,0) {};
		\draw [fill] (1) circle [radius=1.25pt];
		\node [below] at (1) {\small{$C_{2}$}};
		
		\node (2) at (1,0) {};
		\draw [fill] (2) circle [radius=1.25pt];
		\node [below] at (2) {\small{$C_{2}$}};
		
		\node (3) at (2,0) {};
		\draw [fill] (3) circle [radius=1.25pt];
		\node [below] at (3) {\small{$C_{2}$}};
		
		\node (4) at (3,0) {};
		\draw [fill] (4) circle [radius=1.25pt];
		\node [below] at (4) {\small{$C_{2}$}};		
	\end{tikzpicture} \\
	(\emph{Lineal}) & $\Aut(T)_{\omega,\omega'}$ & \begin{tikzpicture}[baseline=(0.base), scale=1.5]
		\node (0) at (0,\ybase) {};	
		\begin{scope}
			[line width=0.6pt, line width=0.6pt, decoration={markings,
				mark=at position 0.5 with {\node[label={[black, label distance=-2mm]{180}:\small{$\{\!1\!\}$}}]{};}					
			}]
			\draw [black, postaction=decorate] (0,0) .. controls (135:0.9) and (180:0.9) .. (0,0);
		\end{scope}
		
		\begin{scope}
			[line width=0.6pt, line width=0.6pt, decoration={markings,
				mark=at position 0.5 with {\node[label={[black, label distance=-2mm]{180}:\small{$\{\!2\!\}$}}]{};}					
			}]
			\draw [black, postaction=decorate] (0,0) .. controls (180:0.9) and (225:0.9) .. (0,0);
		\end{scope}		
		
		\begin{scope}
			[line width=0.6pt, line width=0.6pt, decoration={markings,
				mark=at position 0.53 with {\arrow{>}},
				mark=at position 0.5 with {\node[above]{\small{$\{\!3\!\}$}};}
			}]
			\draw [black, postaction=decorate] (0,0) .. controls (1/4,1/4) and (3/4,1/4) .. (1,0);
		\end{scope}
		
		\begin{scope}
			[line width=0.6pt, line width=0.6pt, decoration={markings,
				mark=at position 0.53 with {\arrow{>}},
				mark=at position 0.5 with {\node[below]{\small{$\{\!1\!\}$}};}
			}]
			\draw [black, postaction=decorate] (1,0) .. controls (3/4,-1/4) and (1/4,-1/4) .. (0,0);
		\end{scope}
		
		\begin{scope}
			[line width=0.6pt, line width=0.6pt, decoration={markings,
				mark=at position 0.53 with {\arrow{>}},
				mark=at position 0.5 with {\node[above]{\small{$\{\!2,\!3\!\}$}};}
			}]
			\draw [black, postaction=decorate] (1,0) .. controls (5/4,1/4) and (7/4,1/4) .. (2,0);
		\end{scope}
		
		\begin{scope}
			[line width=0.6pt, line width=0.6pt, decoration={markings,
				mark=at position 0.53 with {\arrow{>}},
				mark=at position 0.5 with {\node[below]{\small{$\{\!1\!\}$}};}
			}]
			\draw [black, postaction=decorate] (2,0) .. controls (7/4,-1/4) and (5/4,-1/4) .. (1,0);
		\end{scope}
		
		\begin{scope}
			[line width=0.6pt, line width=0.6pt, decoration={markings,
				mark=at position 0.53 with {\arrow{>}},
				mark=at position 0.5 with {\node[above]{\small{$\{\!2,\!3\!\}$}};}
			}]
			\draw [black, postaction=decorate] (2,0) .. controls (9/4,1/4) and (11/4,1/4) .. (3,0);
		\end{scope}
		
		\begin{scope}
			[line width=0.6pt, line width=0.6pt, decoration={markings,
				mark=at position 0.53 with {\arrow{>}},
				mark=at position 0.5 with {\node[below]{\small{$\{\!1\!\}$}};}
			}]
			\draw [black, postaction=decorate] (3,0) .. controls (11/4,-1/4) and (9/4,-1/4) .. (2,0);
		\end{scope}
		
		\draw [black, dashed] (3,0) .. controls (13/4,1/4) and (15/4,1/4) .. (4,0);
		\draw [black, dashed] (3,0) .. controls (13/4,-1/4) and (15/4,-1/4) .. (4,0);
		
		\draw [color=white, fill=white] (4,0.5) rectangle (3.5,-0.5);
		
		\node (1) at (0,0) {};
		\draw [fill] (1) circle [radius=1.25pt];
		\node [below] at (1) {\small{$1$}};
		
		\node (2) at (1,0) {};
		\draw [fill] (2) circle [radius=1.25pt];
		\node [below] at (2) {\small{$C_{2}$}};
		
		\node (3) at (2,0) {};
		\draw [fill] (3) circle [radius=1.25pt];
		\node [below] at (3) {\small{$C_{2}$}};	
		
		\node (4) at (3,0) {};
		\draw [fill] (4) circle [radius=1.25pt];
		\node [below] at (4) {\small{$C_{2}$}};			
	\end{tikzpicture} \\
	(\emph{Focal}) & $\Aut(T)_{\omega}$ & \begin{tikzpicture}[baseline=(0.base), scale=1.5]
		\node (0) at (0,\ybase) {};
		\begin{scope}
			[line width=0.6pt, line width=0.6pt, decoration={markings,
				mark=at position 0.5 with {\node[label={[black, label distance=-2mm]{180}:\small{$\{\!1\!\}$}}]{};}					
			}]
			\draw [black, postaction=decorate] (0,0) .. controls (135:0.9) and (225:0.9) .. (0,0);
		\end{scope}
		
		\begin{scope}
			[line width=0.6pt, line width=0.6pt, decoration={markings,
				mark=at position 0.5 with {\node[label={[black, label distance=-2mm]{0}:\small{$\{\!2,\!3\!\}$}}]{};}					
			}]
			\draw [black, postaction=decorate] (0,0) .. controls (-45:0.9) and (45:0.9) .. (0,0);
		\end{scope}
		
		\node (1) at (0,0) {};
		\draw [fill] (1) circle [radius=1.25pt];
		\node [below] at (1) {\small{$C_{2}$}};		
	\end{tikzpicture} \\
	(\emph{Horocyclic}) & $H$ & \begin{tikzpicture}[baseline=(0.base), scale=1.5]
		\node (0) at (0,\ybase) {};	
		\begin{scope}
			[line width=0.6pt, line width=0.6pt, decoration={markings,
				mark=at position 0.53 with {\arrow{>}},
				mark=at position 0.5 with {\node[above]{\small{$\{\!2,\!3\!\}$}};}
			}]
			\draw [black, postaction=decorate] (0,0) .. controls (1/4,1/4) and (3/4,1/4) .. (1,0);
		\end{scope}
		
		\begin{scope}
			[line width=0.6pt, line width=0.6pt, decoration={markings,
				mark=at position 0.53 with {\arrow{>}},
				mark=at position 0.5 with {\node[below]{\small{$\{\!1\!\}$}};}
			}]
			\draw [black, postaction=decorate] (1,0) .. controls (3/4,-1/4) and (1/4,-1/4) .. (0,0);
		\end{scope}
		
		\begin{scope}
			[line width=0.6pt, line width=0.6pt, decoration={markings,
				mark=at position 0.53 with {\arrow{>}},
				mark=at position 0.5 with {\node[above]{\small{$\{\!2,\!3\!\}$}};}
			}]
			\draw [black, postaction=decorate] (1,0) .. controls (5/4,1/4) and (7/4,1/4) .. (2,0);
		\end{scope}
		
		\begin{scope}
			[line width=0.6pt, line width=0.6pt, decoration={markings,
				mark=at position 0.53 with {\arrow{>}},
				mark=at position 0.5 with {\node[below]{\small{$\{\!1\!\}$}};}
			}]
			\draw [black, postaction=decorate] (2,0) .. controls (7/4,-1/4) and (5/4,-1/4) .. (1,0);
		\end{scope}
		
		\draw [black, dashed] (2,0) .. controls (9/4,1/4) and (11/4,1/4) .. (3,0);
		\draw [black, dashed] (2,0) .. controls (9/4,-1/4) and (11/4,-1/4) .. (3,0);
		
		\draw [color=white, fill=white] (3,0.5) rectangle (2.5,-0.5);
		
		\draw [black, dashed] (0,0) .. controls (-1/4,1/4) and (-3/4,1/4) .. (-1,0);
		\draw [black, dashed] (0,0) .. controls (-1/4,-1/4) and (-3/4,-1/4) .. (-1,0);
		
		\draw [color=white, fill=white] (-1,-0.5) rectangle (-0.5,0.5);
		
		\node (1) at (0,0) {};
		\draw [fill] (1) circle [radius=1.25pt];
		\node [below] at (1) {\small{$C_{2}$}};
		
		\node (2) at (1,0) {};
		\draw [fill] (2) circle [radius=1.25pt];
		\node [below] at (2) {\small{$C_{2}$}};
		
		\node (3) at (2,0) {};
		\draw [fill] (3) circle [radius=1.25pt];
		\node [below] at (3) {\small{$C_{2}$}};		
	\end{tikzpicture} \\
	(\emph{General}) & $\Aut(T)$ & \begin{tikzpicture}[baseline=(0.base), scale=1.5]
		\node (0) at (0,\ybase) {};
		\begin{scope}
			[line width=0.6pt, line width=0.6pt, decoration={markings,
				mark=at position 0.5 with {\node[label={[black, label distance=-2mm]{0}:\small{$\{\!1,\!2,\!3\!\}$}}]{};}					
			}]
			\draw [black, postaction=decorate] (0,0) .. controls (-45:0.9) and (45:0.9) .. (0,0);
		\end{scope}
		
		\node (1) at (0,0) {};
		\draw [fill] (1) circle [radius=1.25pt];
		\node [below] at (1) {\small{$S_{3}$}};		
	\end{tikzpicture}
\end{tabular}
}
\end{example}

Given a tree $T$ and $G\le\Aut(T)$, let $\Type(G)$ denote the action type of $G$. We~show that $\smash{\Type(G)=\Type(\overline{G})=\Type\!\big(G^{(P_{k})}\big)}$ for all $k\in\mathbb{N}$. In particular, $\Type(G)$ is a locally determined global property.

\begin{lemma}\label{lem:p_closure_end_fix}
Let $T$ be a tree, $G\!\le\!\Aut(T)$ and $\omega\!\in\!\partial T$. If $G_{\omega}=G$ then $G^{(P)}_{\omega}=G^{(P)}$.
\end{lemma}

\begin{proof}
Suppose that $G^{(P)}$ does not fix $\omega$. Then there are $a\in AT$ and $h\in G^{(P)}$ such that $a$ is oriented towards $\omega$ but $ha$ is not. By definition of $G^{(P)}$, there is $g\in G$ such that $ha = ga$. Hence $ga$ is not oriented towards $\omega$. That is, $g$ does not fix $\omega$, in contradiction to the assumption.
\end{proof}

\begin{proposition}\label{prop:type_locally_determined}
Let $T$ be a tree and let $G\le\Aut(T)$. Then for all $k\in\mathbb{N}$ we have $\smash{\Type(G)=\Type(\overline{G})=\Type\!\big(G^{(P_{k})}\big)}$.
\end{proposition}	

\begin{proof}
It suffices to show that $\Type(G)=\Type(G^{(P)})$. Then for $k\ge 2$ we have 
\begin{displaymath}
	\Type\!\big(G^{(P_{k})}\big) = \Type\!\left(\!\big(G^{(P_{k})}\big)^{(P)}\right) = \Type\!\big(G^{(P)}\big) = \Type(G)
\end{displaymath}
by \cite[Proposition 3.4 (iv)]{BEW15}. Furthermore, since $G \leq \overline{G} \leq G^{(P)}$ the same result implies $\smash{G^{(P)} \leq \overline{G}^{\scaleto{ (P)}{5pt}} \leq \big(G^{(P)}\big)^{\hspace{-1.5pt}\scaleto{(P)}{5pt}} = G^{(P)}}$ and hence
\begin{displaymath}
	\Type(\overline{G}) = \Type\!\left(\overline{G}^{(P)}\right) = \Type\!\big(G^{(P)}\big) = \Type(G).
\end{displaymath}
		
Suppose now that $G$ is of type (\emph{Fixed vertex}), fixing $v \in VT$. Then for all $\smash{h \in G^{(P)}}$ we have that $hv=gv$ for some $g\in G$. Since $gv=v$ we have that $hv=v$ and conclude that $\smash{G^{(P)}}$ is of type (\emph{Fixed vertex}) as well. 
	
Next, suppose that $G$ is of type (\emph{Inversion}), inverting the arc $a\in AT$. Then for all $h\in G^{(P)}$ we have that $ha = ga$ for some $g \in G$. Since $ga\in\{a,\overline{a}\}$ we have that $ha\in\{a,\overline{a}\}$. Since $\smash{G\le G^{(P)}}$, the group $\smash{G^{(P)}}$ still inverts $a$. We conclude that $G^{(P)}$ is of type (\emph{Inversion}) as well. 
	
If $G$ is of type (\emph{Lineal}), fixing $\omega,\omega'\in\partial T$, then $\smash{G^{(P)}}$ also fixes $\omega,\omega'\in\partial T$ by Lemma~\ref{lem:p_closure_end_fix}. Since $\smash{G\le G^{(P)}}$, the group $\smash{G^{(P)}}$ still translates between $\omega$ and $\omega'$. Hence it does not fix any additional ends. We conclude that $G^{(P)}$ is of type \emph{(Lineal)}. 
	
Suppose now that $G$ is of type (\emph{Focal}), fixing $\omega\in\partial T$. Then $\smash{G^{(P)}}$ also fixes $\omega\in\partial T$ by Lemma~\ref{lem:p_closure_end_fix}. Since $\smash{G\le G^{(P)}}$, the group $\smash{G^{(P)}}$ still includes translations towards $\omega\in\partial T$. Moreover, the group $\smash{G^{(P)}}$ cannot fix any additional end as then $G$ would be of type (\emph{Lineal}). Hence $G^{(P)}$ is of type (\emph{Focal}) as well.
	
Next, suppose that $G$ is of type (\emph{Horocyclic}), fixing $\omega\in\partial T$. Then $\smash{G^{(P)}}$ also fixes $\omega$ by Lemma~\ref{lem:p_closure_end_fix}. Moreover, the group $\smash{G^{(P)}}$ cannot fix another end as then $G$ would be of type (\emph{Lineal}). Since $\smash{G\le G^{(P)}}$, the group $\smash{G^{(P)}}$ also cannot fix a vertex as otherwise $G$ would be of type (\emph{Fixed vertex}). Also, if $\smash{G^{(P)}}$ included a translation $h$ towards $\omega\in\partial T$ then for any arc $a$ on the axis of $h$ there would be $g\in G$ such that $ha=ga$. Hence the arcs $a$ and $ga$ are co-oriented, implying that $g$ is a translation. This contradicts $G$ being of type (\emph{Horocyclic}). Overall, $\smash{G^{(P)}}$ is of type (\emph{Horocyclic}).
	

Finally, if $G$ is of type (\emph{General}), acting with translations and fixing no end, the same is true for $\smash{G^{(P)}}$ as $\smash{G\le G^{(P)}}$. Hence $\smash{G^{(P)}}$ is of type (\emph{General}) as well.
\end{proof}

The group action types discussed above rely on invariant structure --- vertices, edges and ends --- which can be detected from scopos of the associated local action diagram using the following result, whose proof is not made explicit in \cite{RS26}.

\begin{proposition}\label{prop:scopo_structure_correspondence}
Let $\Delta\!=\!(\Gamma, (X_{a}), (G(v)))$ be a local action diagram, $\mathbf{T}\!=\!(T, \pi, \calL)$ a $\Delta$-tree, $G:=\mathbf U(\mathbf{T},(G(v)))$ and $O$ a scopo of $\Delta$.
\begin{enumerate}[(i)]
	\item If $O=O_{\Gamma'}$ is of type (a) then $T':=\pi^{-1}(\Gamma')$ is a subtree of $T$ such that $K(\pi^{-1}(O))=VT'$, and $ T'$ is a $G$-invariant subtree of $ T$. 
	\item If $O$ is of type (b) then there is an end $\xi\in\partial T$ such that $K(\pi^{-1}(O)) = \{\xi\}$, the end $\xi$ is $G$-invariant, and $G$ contains a translation towards $\xi$. 
	\item If $O$ is of type (c) then there is an end $\xi\in\partial T$ such that $K(\pi^{-1}(O)) = \{\xi\}$, the end $\xi$ is $G$-invariant, and $G$ contains no translation towards $\xi$. 
\end{enumerate}
\end{proposition}

\begin{proof}
Suppose that $O=O_{\Gamma'}$ is of type (a) for a cotree $\Gamma'$ of $\Gamma$. Let $T':=\pi^{-1}(\Gamma')$. For all $a \in A\Gamma\backslash A\Gamma'$ we have that either $a\in O$ or $\overline{a}\in O$, and for all $a\in A\Gamma'$ we have that $a, \overline{a} \notin O$. Hence $\pi^{-1}(O) \cap AT'=\varnothing$. Moreover, every $a\in\pi^{-1}(O)$ is oriented towards $T'$. Hence, $f_{\pi^{-1}(O)}(v)=v$ for all $v\in VT'$ and whenever $v\notin VT'$ the sequence $\smash{(v, f_{\pi^{-1}(O)}(v), f_{\pi^{-1}(O)}^{2}(v), \dots)}$ is eventually constant equal to a vertex of $T'$. So $K(\pi^{-1}(O))=VT'$. By Theorem~\ref{thm:attractor}, the tree $T'$ is a cotree of $T$ and $\pi^{-1}(O)=O_{T'}$. As $\pi^{-1}(O)$ is $G$-invariant by Lemma~\ref{lem:scopo_preimage},~so~is~$T'$.
	
Now suppose that $O=O_{\Gamma'}^{+}$ is of type (b). In particular, $O$ is a full orientation of $\Gamma$ and the type of $O$ implies that for all $v\in V\Gamma$ there is a unique $a \in o^{-1}(v)$ with $a \in O$. Since $|X_{a}|=1$ for all $a\in O$, we conclude that for all $\tilde{v} \in VT$ there is a unique $\tilde{a} \in o^{-1}(\tilde{v})$ such that $\tilde{a }\in \pi^{-1}(O)$. Therefore, $f_{\pi^{-1}(O)}(\tilde{v}) = t(\tilde{a})$ for all $\tilde{v} \in VT$ where $\smash{\tilde{a} \in o^{-1}(\tilde{v}) \cap \pi^{-1}(O)}$. Hence the sequence $\smash{(\tilde{v}, f_{\pi^{-1}(O)}(\tilde{v}),f_{\pi^{-1}(O)}^{2}(\tilde{v}), \ldots)}$ is aperiodic for all $\tilde{v}\in VT$ and by Theorem~\ref{thm:attractor} there is an end $\xi$ of $ T$ such that $K(\pi^{-1}(O)) = \{\xi\}$ and $\pi^{-1}(O)=O_{\xi}$. As $O_{\xi}$ is $G$-invariant by Lemma~\ref{lem:scopo_preimage}, so is $\xi$.
	
Now consider the cycle $C$ included in the cotree $\Gamma'$ of $\Delta$. Say $C$ has length $n\in\mathbb{N}$. Let $v\in VC$. Pick $\tilde{v}\in\pi^{-1}(v)$ as well as a line $L$ in $T$ that contains $\tilde{v}$, projects onto $C$ and is coloured periodically of period $n$ in both directions. Then, by definition, $G$ contains a translation of length $n$ towards $\xi$ along $L$ with trivial local actions.
	
Finally, suppose that $O=O_{\xi}$ is of type (c). In particular, $O$ is a full orientation of $\Gamma$, and the type of $O$ implies that for all $v\in V\Gamma$ there is a unique $a \in o^{-1}(v)$ with $a \in O$. Since $|X_{a}|=1$ for all $a\in O$, we conclude that for all $\tilde{v} \in VT$ there is a unique $\tilde{a} \in o^{-1}(\tilde{v})$ such that $\tilde{a }\in \pi^{-1}(O)$. Therefore, $f_{\pi^{-1}(O)}(\tilde{v}) = t(\tilde{a})$ for all $\tilde{v} \in VT$ where $\smash{\tilde{a} \in o^{-1}(\tilde{v}) \cap \pi^{-1}(O)}$. Hence the sequence $\smash{(\tilde{v}, f_{\pi^{-1}(O)}(\tilde{v}),f_{\pi^{-1}(O)}^{2}(\tilde{v}), \ldots)}$ is aperiodic for all $\tilde{v}\in VT$ and by Theorem~\ref{thm:attractor} there is an end $\omega$ of $ T$ such that $K(\pi^{-1}(O)) = \{\omega\}$ and $\pi^{-1}(O)=O_{\omega}$. As $O_{\omega}$ is $G$-invariant by Lemma~\ref{lem:scopo_preimage}, so is~$\omega$.

Suppose $g\in G$ is a translation towards $\omega$. Let $v\in VT$ belong to its translation axis. The ray $R$ starting at $v$ that is oriented towards $\omega$ maps onto the ray in $\Gamma$ that starts at $\pi(v)$ and is oriented towards $\xi$. In particular, no two vertices along $R$ belong to the same $G$-orbit, contradicting the existence of $g$.
\end{proof}

Using Proposition~\ref{prop:scopo_structure_correspondence} we now prove the following characterisation of a $(P)$-closed group's action type in terms of its local action diagram. This statement is included in \cite[Section 5]{RS26} but its proof is not made explicit.

\begin{theorem}\label{thm:lad_group_types}
Let $\Delta = (\Gamma, (G(v)), (X_{a}))$ be a local action diagram, $\mathbf{T} = (T, \pi, \calL)$ a $\Delta$-tree and $G:=\mathbf U(\mathbf{T},(G(v)))\le\Aut_{\pi}(T)$. Then $G$ is of type
\begin{description}[before={{\renewcommand\makelabel[1]{(\emph{##1})}}}]
	\item[Fixed vertex] if and only if $\Gamma$ is a tree and $\Delta$ contains a single vertex cotree. \\ The vertices of $T$ fixed by $G$ correspond to the single vertex cotrees of $\Delta$.
	\item[Inversion] if and only if $\Delta$ contains a (necessarily unique) cotree consisting of a vertex with a non-orientable loop $a\in A\Gamma$ so that $|X_{a}| = 1$. \\ The unique edge of $T$ inverted by $G$ corresponds to this cotree.
	\item[Lineal] if and only if $\Delta$ contains a cyclic cotree $\Gamma'$ with $|X_{a}| = 1$ for all $a \in A\Gamma'$.
	The two ends of $T$ fixed by $G$ correspond to the cyclic orientations of $\Gamma'$.
	\item[Focal] if and only if $\Delta$ contains a cyclic cotree $\Gamma'$ with a cyclic orientation $O\subseteq A\Gamma'$ so that $|X_{a}| = 1$ for all $a \in O$ but there is an $a \in A(\Gamma') \backslash O$ with $|X_{a}| \geq 2$. The unique end of $T$ fixed by $G$ corresponds to the orientation $O$.
	\item[Horocyclic] if and only if $\Gamma$ is a tree and $\Delta$ has a unique horocyclic end. \\ The unique end of $T$ fixed by $G$ corresponds to this horocyclic end.
	\item[General] if and only if it is of none of the types above. \\ There is a unique smallest cotree of $\Delta$ which is not of the form indicating the fixed vertex, inversion, lineal, or focal type. It corresponds to the unique minimal subtree of $T$ on which $G$ acts geometrically dense. Moreover, $\Delta$ does not have any horocyclic ends.
\end{description}
\end{theorem}

\begin{proof}
Suppose that $G$ is of (\emph{Fixed Vertex}) type and let $v\in VT$ be fixed by $G$. If $\Gamma$ is not a tree, then it contains a cycle and hence there is a coloured path in $\Gamma$, starting and ending at $\pi(v)$, which is not the concatenation of a coloured path with a reverse. Thus, by construction of $T$, there is a vertex $w$ in $T$ that is distinct from $v$ but satisfies $\pi(w)=\pi(v)$, contradicting $Gv=\{v\}$. Hence $\Gamma$ is a tree. By Lemma~\ref{lem:scopo_preimage}, the type (a) scopo of $T$ associated to $v$ is the preimage of a scopo of $\Delta$ that comes from a single vertex cotree of $\Delta$ by Proposition~\ref{prop:scopo_structure_correspondence}.

Conversely, suppose that $\Gamma$ is a tree and $\Delta$ contains a single vertex cotree $v$. Applying Proposition~\ref{prop:scopo_structure_correspondence} to the type (a) scopo associated to $v$ we see that $\pi^{-1}(v)$ is a $G$-invariant subtree of $T$. By construction of $T$, the fact that $\Gamma$ is a tree and that $v$ is a cotree implies that $\pi^{-1}(v)$ consists of a single vertex. Therefore, $G$ is of (\emph{Fixed Vertex}) type.

Now suppose that $G$ is of (\emph{Inversion}) type and let $\{a,\overline{a}\}\subseteq AT$ be the inverted edge. By Lemma~\ref{lem:scopo_preimage}, the type (a) scopo of $T$ associated to $\{o(a),t(a)\}$ is the preimage of a scopo of $\Delta$ that comes from a cotree consisting of a vertex with a non-orientable loop labelled by a set of size $1$ due to Proposition~\ref{prop:scopo_structure_correspondence} and the assumption on $G$.

Conversely, suppose $\Gamma'$ is a cotree of $\Gamma$ consisting of a single vertex with a non-orientable loop labelled by a set of size $1$. The scopo associated to $\Gamma'$ is of type (a) and hence $\pi^{-1}(\Gamma')$ is a $G$-invariant subtree of $T$ by Proposition~\ref{prop:scopo_structure_correspondence}. Given that $\Gamma'$ is a cotree, the construction of $T$ implies that $\pi^{-1}(\Gamma')$ consists of two vertices connected by an edge. As $G$ contains an inversion of said edge with trivial local actions we conclude that $G$ is of (\emph{Inversion}) type.

Next, suppose that $G$ is of (\emph{Lineal}) type. The two fixed ends correspond to scopos of type (b) of $\Delta$ by Lemma~\ref{lem:scopo_preimage} and Proposition~\ref{prop:scopo_structure_correspondence}. As $G$ contains a translation between the two fixed ends, the cyclic orientations of the underlying cotrees of said scopos are precisely inverse to each other. The assumption on $G$ implies that every arc in this cycle is labelled by a set of size $1$.

Conversely, suppose that $\Gamma'$ is a cyclic cotree of $\Delta$ and that $|X_{a}|=1$ for all $a\in A\Gamma'$. The two scopos associated to $\Gamma'$ are of type (b) and hence $G$ has two fixed ends by Proposition~\ref{prop:scopo_structure_correspondence}. By definition, $G$ contains a translation between these two ends with trivial local actions. Hence $G$ is of (\emph{Lineal}) type.

Suppose that $G$ is of (\emph{Focal}) type and let $\xi\in\partial T$ be the end fixed by $G$. By Lemma~\ref{lem:scopo_preimage}, the type (c) scopo associated to $\xi$ is the preimage of a scopo of $\Delta$ which is of type (b) by Proposition~\ref{prop:scopo_structure_correspondence} and hence comes from a cyclic cotree $\Gamma'$ of $\Delta$. As $G$ fixes $\xi$, we must have $|X_{a}|=1$ for all $a\in A\Gamma'$ that belong to the cyclic orientation $O$ of $\Gamma'$ induced by $\xi$. Furthermore, as $G$ fixes no other end, there can only be one scopo of type (b) in $\Delta$ and hence, in particular, at least one arc in $A\Gamma'\backslash O$ is labelled by a set of size at least $2$.

Conversely, suppose that $\Gamma'$ is a cyclic cotree of $\Delta$ with the prescribed properties. Then the scopo of $\Delta$ associated to $\Gamma'$ is the unique scopo of type (b) in $\Delta$ hence $G$ is of (\emph{Focal}) type by Proposition~\ref{prop:scopo_structure_correspondence}.

Suppose that $G$ is of (\emph{Horocyclic}) type and let $\xi\in\partial T$ be the unique end fixed by $G$. By Lemma~\ref{lem:scopo_preimage}, the type (c) scopo associated to $\xi$ is the preimage of a scopo $O$ of $\Delta$ which is also of type (c) by Proposition~\ref{prop:scopo_structure_correspondence}. Hence $\Gamma$ is a tree and $O$ comes from the unique horocyclic end of $\Delta$.

Conversely, suppose that $\Gamma$ is a tree and $\Delta$ has a unique horocyclic end $\xi$. Then the scopo associated to $\xi$ is of type (c) by definition and hence $G$ is of (\emph{Horocyclic}) type by Proposition~\ref{prop:scopo_structure_correspondence}.

Finally, suppose that $G$ is of (\emph{General}) type. Then by definition, $G$ is not of any of the other types. To see that there is a unique smallest cotree suppose for a contradiction that there are two distinct smallest cotrees $\Gamma_{1}$ and $\Gamma_{2}$ of $\Gamma$. If $\Gamma_{1}\cap\Gamma_{2}$ is non-empty then the subgraph of $\Gamma$ induced by $V\Gamma_{1}\cap V\Gamma_{2}$ is a smaller cotree of $\Gamma$ contained in both $\Gamma_{1}$ and $\Gamma_{2}$. If $\Gamma_{1}$ and $\Gamma_{2}$ are disjoint, then $\Gamma$ is a tree as neither cotree can contain a cycle. Consider the vertex $v\in V\Gamma_{1}$ that is closest to $\Gamma_{2}$ and the arc $a\in A\Gamma$ with $o(a)=v$ oriented towards $\Gamma_{2}$. As both $\Gamma_{1}$ and $\Gamma_{2}$ are cotrees, and every projecting path from a vertex in the half-tree $T_{\overline{a}}$ to $\Gamma_{2}$ passes through $v$, we see that $\{v\}$ is a cotree, contradicting the type of $G$.
\end{proof}

\section{Discrete (P)-closed groups acting on trees}

In this section we characterise discreteness of $(P)$-closed groups in terms of their local action diagrams. We first show that this is also characterises discreteness for groups that satisfy Tits' original independence property ($P$) and, in fact, the weaker property ($IP_{1}$) of Banks--Elder--Willis \cite[Definition~5.1]{BEW15}.

Given a tree $T$, recall that a group $G\le\Aut(T)$ has Property ($IP_{1}$) if for any arc $a\in AT$ we have $G_{a}=G_{a,T_{a}}\cdot G_{a,T_{\overline{a}}}$, where $T_{a}$ and $T_{\overline{a}}$ denote the half-trees of $T$ defined by $a$ containing $t(a)$ and $o(a)$ respectively.

\begin{proposition}
Let $T$ be a tree and let $G\le\Aut(T)$ satisfy Property~($IP_{1}$). Then $G$ is discrete if and only if $G^{(P)}$ is discrete.
\end{proposition}

\begin{proof}
If $G^{(P)}$ is discrete then $G^{(P)}_{F}$ is trivial for some finite set $F\subseteq VT$. Hence so is $\smash{G_{F}\le G^{(P)}_{F}}$ and thus $G$ is discrete as well. 

Conversely, suppose that $\smash{G^{(P)}}$ is non-discrete. We show that $G$ is non-discrete as well. Let $v\in VT$ and $n\in\mathbb{N}$. Consider the ball $B(v,n)$. By assumption, $\smash{G^{(P)}_{B(v,n)}}$ is non-trivial. Increasing $n$ if necessary, we may assume that there is an arc $a\in AT$ with $o(a)\in S(v,n)$ and $t(a)\in S(v,n\!-\! 1)$, such that there is an element $\smash{h\in G^{(P)}_{B(v,n)}}$ that acts non-trivially on some arc $b\in o^{-1}(o(a))$. Then, by definition of $\smash{G^{\scaleto{(P)}{6pt}}}$, there is an element $g\in G$ such that $g$ and $h$ agree on the arc $b$. As $G$ satisfies Property~($IP_{1}$), we may write $g=g_{a}\cdot g_{\overline{a}}$, where $g_{a}\in G_{a,T_{a}}$ and $g_{\overline{a}}\in G_{a,T_{\overline{a}}}$. Then the element $g_{a}\in G$ acts non-trivially on $b$ and fixes $B(v,n)\subseteq\{a\}\cup T_{a}$. Hence $G$ is non-discrete.
\end{proof}

Unlike the action type introduced in Section~\ref{sec:action_type}, however, discreteness is not a locally determined global property. For example, given a permutation group $F\le S_{d}$, the group $\mathrm{U}_{2}(\Gamma(F))\le\mathrm{Aut}(T_{d})$ introduced in \cite[Section 4.4.1]{Tor23} is $(P_{2})$-closed and discrete. However, passing to its $(P)$-closure $\mathrm{U}_{2}(\Gamma(F))^{(P_{1})}=\mathrm{U}_{1}(F)$ results in a discrete group if and only if $F$ is semiregular by \cite[Proposition 6.20]{GGT18}.

We now turn to characterising discreteness of ($P$)-closed groups.

\begin{theorem}\label{thm:discrete}
Let $\Delta = (\Gamma, (X_{a}), (G(v)))$ be a local action diagram, $\mathbf T = (T, \pi, \calL)$ be a $\Delta$-tree, and $G:=\mathbf U(\mathbf{T},(G(v)))\le\Aut_{\pi}(T)$. If $G$ is of type
\begin{description}[before={{\renewcommand\makelabel[1]{(\emph{##1})}}}]
	\item[Fixed vertex] then $G$ is discrete if and only if $G(v)$ is trivial for almost all $v\in V\Gamma$, and whenever $X_{v}$ ($v\in V\Gamma$) is infinite then $G(v)$ has a finite base and $G(u)$ is trivial for every $u\in V\Gamma$ such that the arc $a\in o^{-1}(v)$ oriented towards $u$ has an infinite colour set.
	\item[Inversion] then $G$ is discrete if and only if $G(v)$ is trivial for almost all $v\in V\Gamma$, and whenever $X_{v}$ ($v\in V\Gamma$) is infinite then $G(v)$ has a finite base and $G(u)$ is trivial for every $u\in V\Gamma$ such that the arc $a\in o^{-1}(v)$ oriented towards $u$ has an infinite colour set.
	\item[Lineal] then $G$ is discrete if and only if $G(v)$ is trivial for all $v\in V\Gamma$.
	\item[Focal] then $G$ is non-discrete.
	\item[Horocyclic] then $G$ is non-discrete.
	\item[General] then $G$ is discrete if and only if $G(v)$ is semiregular for all $v\in V\Gamma'$ and trivial otherwise, where $\Gamma'$ is the unique smallest cotree of $\Delta$.
\end{description}
\end{theorem}

For clarity, the proof of Theorem~\ref{thm:discrete} is split over the following subsections, one for each of the six possible types.

\subsection{The (\emph{Fixed vertex}) Type}

Suppose that $G$ is of (\emph{Fixed Vertex}) type, fixing $\tilde{v}\in VT$. Then by Theorem~\ref{thm:lad_group_types}, the graph $\Gamma$ is a tree and there is single vertex cotree $v\in V\Gamma$ such that $\pi(\tilde{v})=v$. Without loss of generality, assume that the $\Delta$-tree $\mathbf{T}$ was constructed using the base vertex $\tilde{v}\in VT$, as described following Lemma~\ref{lem:delta_tree_construction}.

We first show that $G$ is discrete if the listed conditions are satisfied. By Lemma~\ref{lem:perm_top_discrete}, it suffices to construct a finite set $\smash{\widetilde{F}\subseteq VT}$ so that $\smash{G_{\widetilde{F}}}$ is trivial.

Since $G(w)$ is trivial for almost all $w\in V\Gamma$, there exists $n\in\mathbb{N}$ such that $G(w)$ is trivial for all $w\in V\Gamma$ with $d(w,v)>n$. Iterating over the radius $r\in\{0,\ldots,n\}$ we construct finite vertex sets $V_{r}\subseteq S(v,r)$, finite colour subsets $F_{u}\subseteq X_{u}$ for every $u\in V_{r}$, and finite arc sets $\smash{A_{r}\subseteq\bigcup_{u\in V_{r}}o^{-1}(u)}$. The union of the colour sets $F_{u}$ will give rise to the set $\smash{\widetilde{F}}$. To begin, put $V_{0}:=\{v\}$. Given $a\in A\Gamma$, recall that $\Gamma_{a}$ denotes the half-tree of $\Gamma$ containing $t(a)$. We define $F_{v}\subseteq X_{v}$ to be the union of the set $\{c\in X_{v}\mid \exists u\in V\Gamma_{p(c)}:\ G(u)\text{ is non-trivial}\}$ and a finite base of $G(v)$. Then $F_{v}$ is finite by the assumptions on $\Delta$. Finally, set $A_{0}:=\{a\in o^{-1}(v)\mid F_{v}\cap X_{a}\neq\varnothing\}$.

Now assume that the vertex set $V_{r}$, the colour subsets $\{F_{u}\subseteq X_{u}\mid u\in V_{r}\}$ and the arc sets $A_{r}$ have been defined. We let $V_{r+1}:=\{t(a)\mid a\in A_{r}\}\subseteq S(v,r+1)$. For every $u\in V_{r}$ and $a\in A_{r}\cap o^{-1}(u)$ we make a distinction: if $X_{a}$ is infinite we define $F_{t(a)}:=\varnothing$ and $A_{r+1,a}:=\varnothing$. Otherwise, let $F_{t(a)}\subseteq X_{t(a)}$ be the union of the set $\{c\in X_{t(a)}\mid \exists u\in V\Gamma_{p(c)}:\ G(u) \text{ is non-trivial}\}$ and a finite base of $G(t(a))$ not containing the single colour in $X_{\overline{a}}$. Let $A_{r+1,a}:=\{b\in o^{-1}(t(a))\mid F_{t(a)}\cap X_{b}\neq\varnothing\}$. Finally, set $A_{r+1}:=\bigcup_{a\in A_{r}}A_{r+1,a}$.

Now, consider the set $\smash{F:=\bigcup_{r=0}^{n}\bigcup_{u\in V_{r}}F_{u}\subseteq\bigcup_{v\in V\Gamma}X_{v}}$ and note that it is finite by construction. Let $\smash{\widetilde{F}}\subseteq VT$ be the set of vertices corresponding to coloured paths originating at $v$ containing only colours from the set $F$. Then $\smash{\widetilde{F}}$ is finite as well. By construction, the stabiliser $\smash{G_{\widetilde{F}}}$ induces a trivial local action at every vertex in~$\smash{\widetilde{F}}$. Furthermore, every vertex outside of $\smash{\widetilde{F}}$ falls into one of three cases: either it has distance at least $n+1$ from $\tilde{v}$, it involves a colour not contained in $F$, or it follows a vertex $\tilde{w}\in VT$ such that $G(\pi(\tilde{w}))$ has an infinite orbit in the direction of $\pi(\tilde{v})$. In each case, the stabiliser $\smash{G_{\widetilde{F}}}$ has trivial local action at the respective vertex as well. Overall, $\smash{G_{\widetilde{F}}}$ is trivial.

We now show that if the conditions of the (\emph{Fixed Vertex}) case of Theorem~\ref{thm:discrete} are not satisfied, then $G$ is not discrete.

First, suppose that there are infinitely many vertices $u\in V\Gamma$ such that $G(u)$ is non-trivial. Let $\smash{\widetilde{F}}\subseteq VT$ be any finite subset of vertices. We show that $\smash{G_{\widetilde{F}}}$ is non-trivial. Then $G$ is non-discrete by Lemma~\ref{lem:perm_top_discrete}. Passing to a larger set if necessary, we may assume that $\smash{\widetilde{F}}$ contains the fixed vertex $\smash{\tilde{v}}$. Then the set $\smash{F:=\pi(\widetilde{F})\subseteq V\Gamma}$ is finite as well and hence there exists $w\in V\Gamma\backslash F$ such that $G(w)$ is non-trivial. Let $\calC\in VT$ correspond to a coloured path in $\Delta$ from $v$ to $w$. Since $\{v\}$ is a cotree of $\Delta$ there exists an automorphism $g\in G$ that fixes $\smash{\widetilde{F}}$ and has a non-trivial local action at $\calC\in VT$ by Lemma~\ref{lem:universal_extension}. Hence $G_{\widetilde{F}}$ is non-trivial and $G$ is non-discrete.

Now suppose that there is a vertex $w\in V\Gamma$ such that $X_{w}$ is infinite and $G(w)$ has no finite base. Pick $\tilde{w}\in\pi^{-1}(w)$. Fixing a finite set of vertices $F\subseteq VT$ as well as $\tilde{w}$ forces at most finitely many of the vertices $\{t(a)\mid a\in o^{-1}(\tilde{w})\}$ to be fixed. Hence the local action of $G_{F\cup\{\tilde{w}\}}$ at $\tilde{w}$ is non-trivial as before and $G$ is non-discrete.

Finally, suppose that $X_{u}$ is infinite for some $u\in V\Gamma$ and that there exists $w\in V\Gamma$ such that $G(w)$ is non-trivial and the arc $a\in o^{-1}(u)$ oriented towards $w$ has an infinite colour set. Then the set $\pi^{-1}(w)$ is infinite as well. Since $G(w)$ is non-trivial we may therefore, given any finite set $F$ of vertices of $T$, define a non-trivial automorphism of $T$ fixing $F$ as before.

\subsection{The (\emph{Inversion}) Type}

Suppose $G$ is of (\emph{Inversion}) type, inverting the edge $\{a,\overline{a}\}$ of $T$. Put $\tilde{v}:=o(a)$. Then by Theorem~\ref{thm:lad_group_types}, the local action diagram $\Delta$ contains a unique cotree consisting of a vertex $v\in V\Gamma$, such that $\pi(\tilde{v})=v$ with a non-orientable loop labelled by a set of size $1$. Without loss of generality, assume that the $\Delta$-tree $\mathbf{T}$ was constructed using the base vertex $\tilde{v}\in VT$.

We first show that $G$ is discrete if the listed conditions are satisfied. By Lemma~\ref{lem:perm_top_discrete}, it suffices to construct a finite set $\smash{\widetilde{F}\subseteq VT}$ so that $\smash{G_{\widetilde{F}}}$ is trivial. This construction follows the (\emph{Fixed vertex}) case verbatim except that we need to include the colour of the non-orientable loop attached to $v$ at the first step of the iteration.

Since $G(w)$ is trivial for almost all $w\in V\Gamma$, there exists $n\in\mathbb{N}$ such that $G(w)$ is trivial for all $w\in V\Gamma$ with $d(w,v)>n$. Iterating over the radius $r\in\{0,\ldots,n\}$ we construct finite vertex sets $V_{r}\subseteq S(v,r)$, finite colour subsets $F_{u}\subseteq X_{u}$ for every $u\in V_{r}$ and finite arc sets $\smash{A_{r}\subseteq\bigcup_{u\in V_{r}}o^{-1}(u)}$. The union of the colour sets $F_{u}$ will give rise to the set $\smash{\widetilde{F}}$. To begin, put $V_{0}:=\{v\}$. Given $a\in A\Gamma$, recall that $\Gamma_{a}$ denotes the half-tree of $\Gamma$ containing $t(a)$. We define $F_{v}\subseteq X_{v}$ to be the union of the set $\{c\in X_{v}\mid \exists u\in V\Gamma_{p(c)}:\ G(u)\text{ is non-trivial}\}$ and a finite base of $G(v)$ that contains the colour of the non-orientable loop attached to $v$. Then $F_{v}$ is finite by the assumptions on $\Delta$. Finally, set $A_{0}:=\{a\in o^{-1}(v)\mid F_{v}\cap X_{a}\neq\varnothing\}$.

Now assume that the vertex set $V_{r}$, the colour subsets $\{F_{u}\subseteq X_{u}\mid u\in V_{r}\}$ and the arc sets $A_{r}$ have been defined. We let $V_{r+1}:=\{t(a)\mid a\in A_{r}\}\subseteq S(v,r+1)$. For every $u\in V_{r}$ and $a\in A_{r}\cap o^{-1}(u)$ we make a distinction: if $X_{a}$ is infinite we define $F_{t(a)}:=\varnothing$ and $A_{r+1,a}:=\varnothing$. Otherwise, let $F_{t(a)}\subseteq X_{t(a)}$ be the union of the set $\{c\in X_{t(a)}\mid \exists u\in V\Gamma_{p(c)}:\ G(u) \text{ is non-trivial}\}$ and a finite base of $G(t(a))$ not containing the single colour in $X_{\overline{a}}$. Let $A_{r+1,a}:=\{b\in o^{-1}(t(a))\mid F_{t(a)}\cap X_{b}\neq\varnothing\}$. Finally, set $A_{r+1}:=\bigcup_{a\in A_{r}}A_{r+1,a}$.

Now, consider the set $\smash{F:=\bigcup_{r=0}^{n}\bigcup_{u\in V_{r}}F_{u}\subseteq\bigcup_{v\in V\Gamma}X_{v}}$ and note that it is finite by construction. Let $\smash{\widetilde{F}}\subseteq VT$ be the set of vertices corresponding to coloured paths originating at $v$ containing only colours from the set $F$. Then $\smash{\widetilde{F}}$ is finite as well. By construction, the stabiliser $\smash{G_{\widetilde{F}}}$ induces a trivial local action at every vertex in~$\smash{\widetilde{F}}$. Furthermore, every vertex outside of $\smash{\widetilde{F}}$ falls into one of three cases: either it has distance at least $n+1$ from $\tilde{v}$, it involves a colour not contained in $F$, or it follows a vertex $\tilde{w}\in VT$ such that $G(\pi(\tilde{w}))$ has an infinite orbit in the direction of $\pi(\tilde{v})$. In each case, the stabiliser $\smash{G_{\widetilde{F}}}$ has trivial local action at the respective vertex as well. Overall, $\smash{G_{\widetilde{F}}}$ is trivial.

We now show that if the conditions of the (\emph{Inversion}) case of Theorem~\ref{thm:discrete} are not satisfied then $G$ is not discrete.

First, suppose that there are infinitely many vertices $u\in V\Gamma$ such that $G(u)$ is non-trivial. Let $\smash{\widetilde{F}\subseteq VT}$ be any finite set of vertices. We show that $G_{\widetilde{F}}$ is non-trivial. Then $G$ is non-discrete by Lemma~\ref{lem:perm_top_discrete}. Passing to a larger set if necessary, we may assume that $\tilde{F}$ contains the inverted edge $\{a,\overline{a}\}$. The set $\smash{F:=\pi(\widetilde{F})\subseteq V\Gamma}$ is finite as well and hence there exists $w\in V\Gamma\backslash F$ such that $G(w)$ is non-trivial. Let $\calC\in VT$ correspond to a coloured path in $\Delta$ from $o(a)$ to $w$. Since $\pi(o(a))$ belongs to the cotree of $\Delta$ we may define an automorphism $g\in G$ that fixes $\smash{\widetilde{F}}$ and has a non-trivial local action at $\calC\in VT$ by Lemma~\ref{lem:universal_extension}. Hence $G_{\widetilde{F}}$ is non-trivial and $G$ is non-discrete.

Now suppose that there is a vertex $w\in V\Gamma$ such that $X_{w}$ is infinite and $G(w)$ has no finite base. Pick $\smash{\tilde{w}}\in\pi^{-1}(w)$. Fixing a finite set of vertices $F\subseteq VT$ as well as $\smash{\tilde{w}}$ forces at most finitely many vertices $\{t(a)\mid a\in o^{-1}(\tilde{w})\}$ to be fixed. Hence the local action of $G_{F\cup\{\tilde{w}\}}$ at $\tilde{w}$ is non-trivial as before and $G$ is non-discrete.

Finally, suppose that $X_{u}$ is infinite for some $u\in V\Gamma$ and that there exists $w\in V\Gamma$ such that $G(w)$ is non-trivial and the arc $a\in o^{-1}(u)$ oriented towards $w$ has an infinite colour set. Then the set $\pi^{-1}(w)$ is infinite as well. Since $G(w)$ is non-trivial we may therefore, given any finite set $F$ of vertices of $T$, define a non-trivial automorphism of $T$ fixing $F$ as before.

\subsection{The (\emph{Lineal}) Type}

Suppose that $G$ is of (\emph{Lineal}) type. Then by Theorem~\ref{thm:lad_group_types} there is a cyclic cotree $\Gamma'$ in $\Delta$ such that $|X_{a}|=1$ for all $a\in A\Gamma'$.

We first show that $G$ is discrete if $G(v)$ is trivial for all $v\in V\Gamma$. Pick any $w\in VT$. Since every element of $G_{w}$ fixes $w$ and has trivial local actions by assumption we conclude that $G_{w}$ is trivial and hence $G$ is discrete by Lemma~{\ref{lem:perm_top_discrete}.
	
Conversely, suppose that $G(v)$ is non-trivial for some $v\in V\Gamma$. Since $\Gamma'$ is a cotree, $G(v)$ is in fact non-trivial on colours whose arcs are oriented away from $\Gamma'$. Moreover, since $\Gamma'$ is cyclic, for every finite set of vertices $F\subseteq VT$ there is $\tilde{v}\in\pi^{-1}(v)$ not contained in the subtree induced by $F$ and we may define an element of $G_{F}$ with non-trivial local action at $\tilde{v}$ by Lemma~\ref{lem:universal_extension}.

\subsection{The (\emph{Focal}) Type}

Suppose that $G$ is of (\emph{Focal}) type. Then by Theorem~\ref{thm:lad_group_types} the local action diagram $\Delta$ contains a cyclic cotree $\Gamma'$ with a cyclic orientation $O\subseteq A\Gamma'$ such that $|X_{a}|=1$ for all $a\in O$, representing the unique fixed end $\xi\in\partial T$, but $|X_{a}|>1$ for at least one $a\in A\Gamma'\backslash O$. In particular, there is $w\in V\Gamma'$ such that $G(w)$ is non-trivial, has an orbit of size greater than $1$ along $\overline{O}$ and an orbit of size $1$ along $O$.

Without loss of generality, suppose that the construction of $T$ began at $\tilde{v}\in VT$. Given any finite set $F\subseteq VT$, we construct a vertex $\calC\in VT$ such that $G$ admits a non-trivial local action on arcs originating at $\calC$ that are oriented away from $F$ and~$\xi$. There exists an element $g\in G_{F}$ with non-trivial local action at $\calC$ by Lemma~\ref{lem:universal_extension}. Let $n\in\mathbb{N}$ be such that the subtree of $T$ induced by $F$ is contained in $B_{n}(\tilde{v})$. Starting from the empty path, add colours to $\calC$ by following $\overline{O}$ until at least $n$ colours have been added and the last colour comes from the arc terminating at $w$. In $T$, the corresponding path ends with an arc oriented away from both $\xi$ and $\tilde{F}$ and such that $G(\pi(\calC))=G(w)$ admits a non-trivial local action on colours whose arc is oriented away from $\xi$. Hence we may define $g$ as described and so $G$ is non-discrete by Lemma~\ref{lem:perm_top_discrete}.

\subsection{The (\emph{Horocyclic}) Type}

Suppose that $G$ is of (\emph{Horocyclic}) type. Then by Theorem~\ref{thm:lad_group_types} the graph $\Gamma$ is a tree and $\Delta$ has a unique horocyclic end $\xi$. Let $R$ be any ray in $\Gamma$ representing $\xi$. Since $\xi$ is horocyclic and $G$ is not of (\emph{Fixed vertex}) type, there must be infinitely many arcs $a$ belonging to $R$, oriented away from $\xi$, such that $|X_{a}|>1$; otherwise, $R$ would contain a single vertex cotree and so $G$ would be of (\emph{Fixed vertex}) type. As a consequence, there are infinitely many vertices $w$ belonging to $R$ such that $G(w)$ is non-trivial.

Now, let $F\subseteq VT$ be any finite set of vertices. We construct a vertex $\calC\in VT$, such that $G$ admits a non-trivial local action on arcs originating at $\calC$ that are oriented away from $F$ and $\xi$. We may then define an element $g\in G_{F}$ with non-trivial local action at $\calC$ by Lemma~\ref{lem:universal_extension}. Suppose that the construction of $T$ started at $\tilde{v}\in VT$. Let $n\in\mathbb{N}$ be such that the subtree of $T$ induced by $F$ is contained in $B_{n}(\tilde{v})$. Let $\tilde{R}$ be the ray in $T$ that starts at $\tilde{v}$ and is oriented towards the fixed end. Then $R:=\pi(\tilde{R})$ is a ray representing $\xi$. By the previous paragraph there are vertices $w_{1}$ and $w_{2}$ in $R$ at distance $n_{1}$ and $n_{2}$ respectively from $v:=\pi(\tilde{v})$ such that $n<n_{1}<n_{2}$, and both $G(w_{1})$ and $G(w_{2})$ are non-trivial on arcs oriented away from $\xi$. Define the coloured path $\calC$ by adding colours to get from $v$ to $w_{2}$ and, using the fact that $|X_{a}|\ge 2$ for $a\in o^{-1}(w_{2})$ oriented towards $v$, extend the path with colours to terminate at $w_{1}$. In $T$, the corresponding path ends with an arc oriented away from both $\xi$ and $\tilde{F}$ and such that $G(\pi(\calC))=G(w_{1})$ admits a non-trivial local action on colours whose arcs are oriented away from $\xi$. Hence we may define $g$ as described and so $G$ is non-discrete by Lemma~\ref{lem:perm_top_discrete}.

\subsection{The (\emph{General}) Type}

Suppose that $G$ is of (\emph{General}) type. Then by Theorem~\ref{thm:lad_group_types} the local action diagram $\Delta$ contains a unique smallest cotree $\Gamma'$ that is not of the form indicating any of the other types. In particular, $\Gamma'$ corresponds to a scopo of type (a) and contains more than one vertex. From Proposition~\ref{prop:scopo_structure_correspondence} we conclude that $\pi^{-1}(\Gamma')$ is a $G$-invariant subtree of $T$ containing at least two vertices.

If the assumptions of the statement are satisfied then, in particular, every local action of $G$ is either semiregular or trivial. Given any arc $a\in AT$, we conclude that every $g\in G_{a}$ acts trivially on $B_{n}(o(a))\cup B_{n}(t(a))$ for every $n\in\bbN$ by induction. Hence $G_{a}$ is trivial and so $G$ is discrete by Lemma~\ref{lem:perm_top_discrete}.

Conversely, suppose that either there is a vertex $w\in V\Gamma'$ such that $G(w)$ is not semiregular, or a vertex $w\in V\Gamma\backslash V\Gamma'$ such that $G(w)$ is non-trivial. Given any finite set of vertices $F\subseteq VT$ we construct a non-trivial element $g\in G_{F}$. Our argument is based on the following lemma.

\begin{lemma}\label{lem:general_type_discrete}
Let $\Delta=(\Gamma,(X_{a}),(G(v)))$ be a local action diagram with associated group of (\emph{General}) type. Further, let $\Gamma'$ be the unique smallest cotree, and $n\in\mathbb{N}$. Then for any two vertices $v, w\in V\Gamma'$ and $a\in o^{-1}(w)\cap A\Gamma'$ there is a coloured path of length greater than $n$ representing a vertex of a $\Delta$-tree $\mathbf{T}=(T,\pi,\calL)$ constructed from $v$, consisting of colours from $\Gamma'$, and whose last colour belongs to $X_{\overline{a}}$.
\end{lemma}

\begin{proof}
Since $\Gamma'$ is connected as a cotree in the connected graph $\Gamma$, there is a coloured path from $v$ to $w$ within $\Gamma'$, for any $v,w\in V\Gamma'$, which represents a vertex of $T$. If the reverse of some arc in such a path is labelled by a set of size greater than $1$, we may construct an arbitrarily long such coloured path by repeating said arc and its reverse. Suppose therefore that for some $w\in V\Gamma'$ and $a\in o^{-1}(w)\cap A\Gamma'$ there is no such arc on any coloured path from $v$ to $w$ terminating in a colour from $X_{\overline{a}}$. Then, in particular, $|X_{a}|=1$ for all $a\in o^{-1}(w)$. If $w$ is not a cut vertex of $\Gamma'$ we therefore obtain a smaller cotree by removing $w$ from $\Gamma'$, contradicting the minimality of $\Gamma'$. If $w$ is a cut vertex of $\Gamma'$, let $\Gamma_{1}'$ and $\Gamma_{2}'$ be the two connected components obtained from $\Gamma'$ by removing $w$. Say $v\in V\Gamma_{1}'$. The nonexistence of a path from $v$ to $w$ with the desired properties implies that $|X_{a}|=1$ for all $a\in A\Gamma_{1}'$. Again, we obtain a smaller cotree contradicting the minimality of $\Gamma'$.
\end{proof}

\vspace{-0.02cm}
Now, let $F\subseteq VT$ be a finite set of vertices. Suppose that the construction of $T$ started at $\tilde{v}\in VT$ and let $n\in\bbN$ be such that the subtree of $T$ induced by $F$ is contained in $B_{n}(\tilde{v})$. As a consequence of Lemma~\ref{lem:general_type_discrete}, there is, for any vertex $w\in V\Gamma$, a coloured path connecting $v:=\pi(\tilde{v})$ to $w$ of length greater than $n$, terminating with any chosen incoming arc of $w$ and representing a vertex of $T$.

Suppose there is a vertex $w\in V\Gamma'$ such that $G(w)$ is not semiregular. In particular, there is a colour $c\in X_{w}$ such that $G(w)_{c}$ is non-trivial, say $c\in X_{a}$ for some $a\in o^{-1}(w)$. Using the above, construct a coloured path $\calC$ of length greater than $n$, terminating at $w$ with the reverse of $a$ and with final reverse color $c$. By Lemma~\ref{lem:universal_extension}, we may define an element $g\in G_{F}$ with non-trivial local action at $\calC$, fixing the arc labelled $c$ originating at $\calC$ that is oriented towards $F$. Hence $G$ is non-discrete.

Similarly, suppose there is a vertex $w\in V\Gamma$ such that $G(w)$ is not trivial. Since $\Gamma'$ is a cotree, every arc originating at $w$ which is oriented towards $\Gamma'$ is labelled by a set of size $1$, i.e. a fixed point of $G(w)$, and so there is an arc $a\in o^{-1}(w)$, oriented away from $\Gamma'$ such that $G(w)$ acts non-trivially on $X_{a}$. Proceed as above to construct a coloured path $\calC$ of length greater than $n$ terminating at $w$ with the reverse of an arc oriented towards $\Gamma'$. Then, as before, we may define an element $g\in G_{F}$ with non-trivial local action at $\calC$, fixing the arc originating at $\calC$ that is oriented towards $F$. Hence $G$ is non-discrete.

\bibliographystyle{amsalpha}
\bibliography{discrete}

\end{document}